\documentclass{amsart}
  
  \usepackage{amsmath,amssymb,amsthm}
 \usepackage{graphicx}
  \usepackage{enumerate}
  \usepackage{xspace}
  \usepackage[all]{xy}
  \usepackage{pinlabel}

  \usepackage{hyperref}

  \newcommand{\calC}{\mathcal{C}}

  \newcommand{\calF}{\mathcal{F}}
  \newcommand{\calG}{\mathcal{G}}

  \newcommand{\calL}{\mathcal{L}}
  \newcommand{\calM}{\mathcal{M}}

  \newcommand{\calP}{\mathcal{P}}
  \newcommand{\calQ}{\mathcal{Q}}

  \newcommand{\calT}{\mathcal{T}}

  \newcommand{\RR}{\mathbb{R}}

  \newcommand{\ZZ}{\mathbb{Z}}

  \newcommand{\gothic}{\mathfrak}
  \newcommand{\Ga}{{\gothic a}}
  \newcommand{\Gb}{{\gothic b}}

  \newtheorem{theorem}{Theorem}[section]
  \newtheorem{proposition}[theorem]{Proposition}
  \newtheorem{corollary}[theorem]{Corollary}
  \newtheorem{lemma}[theorem]{Lemma}

  \theoremstyle{definition}
  \newtheorem{definition}[theorem]{Definition}
  \newtheorem{claim}[theorem]{Claim}
  
  \newtheorem*{claim*}{Claim}

  \newtheorem*{question*}{Question}
  \newtheorem*{answer*}{Answer}
  \newtheorem*{application*}{Application}

  \theoremstyle{remark}
  \newtheorem{remark}[theorem]{Remark}
  \newtheorem*{remark*}{Remark}

  \newcommand{\secref}[1]{Section~\ref{Sec:#1}}
  \newcommand{\thmref}[1]{Theorem~\ref{Thm:#1}}
  \newcommand{\corref}[1]{Corollary~\ref{Cor:#1}}
  \newcommand{\lemref}[1]{Lemma~\ref{Lem:#1}}
  \newcommand{\propref}[1]{Proposition~\ref{Prop:#1}}

  \newcommand{\figref}[1]{Figure~\ref{Fig:#1}}
  \newcommand{\defref}[1]{Definition~\ref{Def:#1}}
  
  \newcommand{\eqnref}[1]{Equation~\eqref{Eq:#1}}

  \DeclareMathOperator{\twist}{twist}
  \DeclareMathOperator{\width}{width}

  \DeclareMathOperator{\Hyp}{Hyp}

  \DeclareMathOperator{\I}{i}

  \newcommand{\emul}{\stackrel{{}_\ast}{\asymp}}
  \newcommand{\gmul}{\stackrel{{}_\ast}{\succ}}
  \newcommand{\lmul}{\stackrel{{}_\ast}{\prec}}
  \newcommand{\eadd}{\stackrel{{}_+}{\asymp}}
  \newcommand{\gadd}{\stackrel{{}_+}{\succ}}
  \newcommand{\ladd}{\stackrel{{}_+}{\prec}}

  \newcommand{\T}{\ensuremath{\mathcal{T}}\xspace} 
   
  \newcommand{\ML}{\ensuremath{\mathcal{ML}}\xspace} 
   
  \newcommand{\EL}{\ensuremath{\mathcal{EL}}\xspace} 
  \newcommand{\PML}{\ensuremath{\mathcal{PML}}\xspace}  

  \newcommand{\Teich}{{Teichm\"uller }} 
  \newcommand{\K}{{\sf K}}

  \newcommand{\param}{{\mathchoice{\mkern1mu\mbox{\raise2.2pt\hbox{$
  \centerdot$}}
  \mkern1mu}{\mkern1mu\mbox{\raise2.2pt\hbox{$\centerdot$}}\mkern1mu}{
  \mkern1.5mu\centerdot\mkern1.5mu}{\mkern1.5mu\centerdot\mkern1.5mu}}}

  \newcommand{\from}{\colon\thinspace} 
  \newcommand{\ep}{\epsilon}

  \newcommand{\dT}{{d_T}}

  \newcommand{\bnu}{{\overline \nu}}    
  \newcommand{\ba}{{\overline a}}
  \newcommand{\bb}{{\overline b}}

  \newcommand{\ua}{{\underline a}}
  \newcommand{\ub}{{\underline b}}

\begin{document}


  \title[Limit sets of Teichm\"uller geodesics]   
  {Limit sets of Teichm\"uller geodesics with
  minimal non-uniquely ergodic vertical foliation.}
  \author{Christopher Leininger} 
  \author{Anna Lenzhen}
  \author   {Kasra Rafi}
   \thanks{The first author was partially supported by NSF grant DMS-1207183 and also acknowledges support from NSF grants DMS 1107452, 1107263, 1107367 ``RNMS: Geometric structures And Representation varieties" (the GEAR Network). The second author was partially supported by ANR grant GeoDyM}
  \date{\today}

\begin{abstract} 
  We describe a method for constructing \Teich geodesics where 
  the vertical foliation $\nu$ is minimal but is not 
  uniquely ergodic and where we have a good understanding
  of the behavior of the \Teich geodesic.  The construction depends
  on various parameters, and we show that
  one can adjust the parameters to ensure
  that the set of accumulation points of such a geodesic in the 
  Thurston boundary is exactly the projective $1$--simplex of all projective measured 
  foliations that are topologically equivalent to $\nu$. With further adjustment
  of the parameters, one can further assume that the transverse measure is an ergodic measure on the non-uniquely ergodic foliation $\nu$.
\end{abstract}
  
  \maketitle
  

\section{Introduction}

In this paper, we describe a family of geodesic laminations on a surface. These are 
limits of explicit sequences of simple closed curves on the surface that
form a quasi-geodesic in the curve complex and hence,
by a theorem of Klarreich \cite{klarreich:bc}, are minimal, filling
and measurable.  Analogous to continued fraction coefficients associated to 
an irrational real number, any lamination $\nu$ in our family has an associated infinite 
sequence of positive integers $\{r_i\}$  
that can be chosen essentially arbitrarily and encode the \emph{arithmetic properties} 
of $\nu$.  One important consequence of the explicit nature of our construction is that 
we can adjust the values of $\{r_i\}$ to produce different interesting examples. 
In particular, we show that if the sequence $r_i$ grows fast enough then
the lamination $\nu$ is not uniquely ergodic. 

Although our construction is general in spirit, we carry out detailed computations 
in the case of the five-times punctured sphere. This case is already 
rich enough for us to observe some interesting phenomena. 
Our lamination $\nu$ is a limit of a sequence of simple closed curves $\gamma_i$
defined as follows.  In \secref{Construction} we describe a finite order homeomorphism $\rho$ of a five-times punctured sphere $S$ and a particular Dehn twist $D$.  We then set $\phi_r = D^r \circ \rho$ and define
\[
\Phi_i = \phi_{r_1} \circ \ldots \circ \phi_{r_i}
\qquad\text{and}\qquad \gamma_i = \Phi_i(\gamma_0),
\]
(see \secref{Construction} for detailed description). 

\begin{theorem}\label{Thm:NUE}
There exists $R >0$ so that if the powers $r_i$ are larger than $R$, then the path
$\{\gamma_i\}$ is a quasi-geodesic in the curve complex
and hence the limiting lamination exists, is minimal and filling.
Furthermore, if $r_i$ grow fast enough---specifically if
$r_i \leq \epsilon r_{i+1}$ for some $\epsilon < \tfrac12$ and all $i \geq 0$---then the limiting lamination $\nu$ 
is not uniquely ergodic. 
\end{theorem}

We are interested in understanding a  \Teich geodesic
where $\nu$ is topologically equivalent to the vertical foliation of
the associated quadratic differential. (Recall that there is a 
one-to-one correspondence between measured laminations and 
singular measured foliations; see Section \ref{S:MLMF}).  
Let $\Delta(\nu) \subset \PML(S)$ be the simplex of all possible 
projectivized measures on $\nu$. In the case of the five-times punctured 
sphere, if $\nu$ is minimal and not uniquely ergodic then $\Delta(\nu)$
is one dimensional, that is, it is homeomorphic to an interval. 
The endpoints of this interval are projective classes associated to ergodic
measures on $\nu$, denoted $\bnu_\alpha$ and
$\bnu_\beta$. Every other measure takes the form 
$\bnu = c_\alpha \bnu_\alpha + c_\beta \bnu_\beta$ for positive 
real numbers $c_\alpha$ and $c_\beta$. Note that the projective class of 
$\bnu$ depends only on the ratio of $c_\alpha$ and $c_\beta$.

Now fix any point $X$ in \Teich space. By a theorem
of Hubbard-Masur \cite{masur:QDF}, there is a unique \Teich geodesic 
\[g = g (X, \bnu) \from [0, \infty) \to \T(S)\]
starting from $X$ so that the vertical foliation associated to $g$ is 
in the projective class of $\bnu$.
We examine the limit set $\Lambda(g)$ of this geodesic in the Thurston 
boundary of \Teich space. 

By appealing to the results of the third author in \cite{rafi:SC, rafi:HT}, we 
can determine how the coefficients $\{r_i\}$ effect the behavior of this \Teich 
geodesic. In particular, at any time $t$, we can describe the geometry of the 
surface $g(t)$ using the numbers $r_i$.  As a result, we can control the limit 
set of $g$.

\begin{theorem}\label{Thm:intro}
Let $\bnu = c_\alpha \bnu_\alpha + c_\beta \bnu_\beta$ 
and $g = g(X, \bnu)$ be as above. If $\{r_i\}$ satisfies
certain growth conditions, (see \defref{Growth} where
$n_i = r_{2i-1}$ and $m_i = r_{2i}$) then 
\[\Lambda(g) = \Delta(\nu),\]
for any value of $c_\alpha$ and $c_\beta$.
\end{theorem}

Note that, in particular, even when $\bnu= \bnu_\alpha$
is an ergodic measure, the limit set still includes the other ergodic measure
$\bnu_\beta$. This contrasts with the work of Lenzhen-Masur, where they
show that the geodesics $g_\alpha = g(X, \bnu_\alpha)$ and $g_\beta = g(X, \bnu_\beta)$
diverge in \Teich space, that is the distance 
in \Teich space between $g_\alpha(t)$ and $g_\beta(t)$ goes to 
infinity as $t \to \infty$ \cite{lenzhen:CD}.


\subsection*{History and related results}
The existence of minimal, filling and non-uniquely ergodic foliations
has been known for a long time due to work of Keane \cite{keane:NI}, 
Sataev \cite{sataev:NIM}, Keynes-Newton \cite{newton:NUE},
Veech \cite{veech:SE} and Masur-Tabachnikov \cite{masur:RB}. 
These constructions generally use a  flat structure on a surface where the 
associated measured foliation in certain directions are non-uniquely ergodic. 
In fact, much in known about the prevalence of such foliations
(see for example results by Masur \cite{masur:HD}, 
Masure-Smillie \cite{masur:HDN}, Cheung \cite{cheung:HD},  
Cheung-Masur \cite{masur:MN}, Cheung-Hubert-Masur \cite{masur:TD}, and Athreya-Chaika \cite{athreya-chaika:HD} 
about the Hausdorff dimension of the set non-uniquely ergodic foliations in 
various contexts) and the divergence rate of a \Teich geodesic with
a non-uniquely egodic vertical foliation (see Masur \cite{masur:HD}, Cheung-Eskin \cite{eskin:UE}, Trevi\~no \cite{trevino:RATE} ). In contrast, our construction is topological
and is similar to Gabai's construction of non-uniquely ergodic 
arational measured foliations \cite{gabai:AF}, though our analysis is quite different from Gabai's.  A similar construction was also known
to Brock \cite{brock:CF} who studied the Weil-Petersson geodesics associated 
$\nu$. In \cite{brockmodami}, Brock and Modami show that such a geodesic provides the counter example for
a \emph{Masur criterion} in the setting of Weil-Petersson metric.  That is, even 
though the lamination is not uniquely ergodic, the Weil-Petersson geodesic
does not diverge in the moduli space. 

The limit set of a Teichm\"uller ray in the Thurston boundary has also been
studied before. The question is non trivial since a  \Teich ray is defined by 
deforming a flat structure, while Thurston compactification is defined via hyperbolic 
geometry, that is, one needs to know how the hyperbolic metric changes along
the ray to be able to say something about its limit set. 

From the work of Masur (\cite{masur:TB}, 1982), we know that in most directions
 geodesic rays converge to a point in the Thurston boundary. More precisely, if 
 the associated vertical measured foliation is uniquely ergodic, that is, has no closed 
 leaves and supports a unique up to scaling invariant measure, then
 the ray converges to the projective class of this foliation. Another case that is 
 completely solved in  \cite{masur:TB} is when the vertical measured foliation $\bnu$ determining 
 the ray is Strebel, that is when the complement of singular leaves is a union of 
 cylinders. Any such ray 
 converges to a point, which is homeomorphic to $\bnu$, but the transverse measure 
 is such that all cylinders have equal heights.
  
The first example of a ray that does not converge to a unique point was given 
by the second author in 
\cite{lenzhen:GL} for a genus $2$ surface. The associated vertical foliation 
in that example has 
a closed singular leaf and two minimal components. In particular, the vertical 
foliation is not minimal. 

If a \Teich ray is defined by a vertical measured foliation $\bnu$ which is minimal, it is known that
any point in its limit set is topologically equivalent to $\bnu$, that is
$\Lambda(g) \subset \Delta(\nu)$.  Until now,
nothing was known about how big the limit set can be in the case of minimal,
filling laminations. Chaika-Masur-Wolf have also constructed examples
where the limit set of a ray can be more than one point \cite{chaikamasurwolf:IP}.
They consider the Veech example of two tori glued along a slit with irrational length 
that gives non-ergodic examples in various directions and investigate the limit in 
$\PML(S)$. They produce both ergodic and nonergodic examples where the limit set is an interval and such examples where the limit set is a point.  Moreover, in the latter case they identify which point it is---either the ergodic measure or the barycenter.

\subsection{Notation}
Let $X$ be a Riemann surface. We denote the hyperbolic length of
a curve $\alpha$ in $X$ by $\Hyp_X(\alpha)$. A quadratic differential 
on $X$ is denoted by $(X, q)$. The flat length of $\alpha$ on $(X,q)$ is denoted
by $\ell_q(\alpha)$. A \Teich geodesic ray $g \from \RR_+ \to \T(S)$ defines 
a one-parameter family of quadratic differentials $(X_t,q_t$). To refer to the hyperbolic 
length and the flat length of a curve $\alpha$ at time $t$, we often use the notation 
$\Hyp_t(\alpha)$ and $\ell_t(\alpha)$ instead of $\Hyp_{X_t}(\alpha)$
and $\ell_{q_t}(\alpha)$ respectively. 

The notation $\emul$ means equal up to multiplicative error and $\eadd$ means 
equal up to an additive error with uniform constants. For example
\[
\Ga  \emul \Gb  \quad \Longleftrightarrow\quad
\frac{\Gb}K \leq \Ga \leq K \, \Gb, \qquad
\text{for a uniform constant $K$}. 
\]
When $\Ga \emul \Gb$, we say $\Ga$ and $\Gb$ are \emph{comparable}. 
The notations $\gmul$ and $\gadd$ are similarly defined. 

\subsection{Acknowledgements}  The authors would like to thank Howard Masur and Jeffrey Brock for useful conversations related to this work.  They would also like to thank the referee for carefully reading the paper and providing many useful suggestions.

\section{Background}

\subsection{ \Teich space}
Let $S$ be the five-times punctured sphere. 
 The \Teich space $\calT(S)$ is the space of marked conformal structures on $S$
  up to isotopy.  Via uniformization, $T(S)$ can be viewed as a space of finite area, complete, hyperbolic
  metrics up to isotopy.
  
  In this paper, we consider  \Teich metric on $\calT(S)$. 
  Given  $x, y \in \T$, the \emph{\Teich distance}
  between them is defined to be \[ \dT(x,y) = \frac{1}{2} \inf_f \log
  K(f),\] where $f \from x \to y$ is a $K(f)$--quasi-conformal homeomorphism
  preserving the marking. (See \cite{gardiner:QT} and \cite{hubbard:TT} for
  background information.) Geodesics in this metric are called \Teich geodesics.
  We will briefly describe them.
  
 Let $X\in\calT(S)$ and $q=q(z)dz^2$ be a quadratic differential on $X$. There
 exists a  \emph{natural parameter} $\zeta=\xi+i\eta$, 
which is defined away from its singularities as
$$\zeta(w)=\int_{z_0}^{w}\sqrt{q(z)}\, dz.$$
In these coordinates, we have $q=d\zeta^2$. 
The lines $\xi=const$ with transverse measure $|d\xi|$ define the 
\textit{vertical} measured foliation, associated to $q$. Similarly, 
the \textit{horizontal} measured foliation is defined by $\eta=const$  and 
$|d\eta|$.  The transverse measure of an arc $\alpha$ with respect to 
$|d\xi|$, denoted by $h_q(\alpha)$, is called the \textit{horizontal length} of 
$\alpha$. Similarly, the \textit{vertical length} $v_q(\alpha)$ is the measure 
of $\alpha$ with respect to $|d\eta|$.

A \Teich geodesic can be described as follows. Given a Riemann surface
$X_0$ and a quadratic differential $q$ on $X_0$, we can obtain a
$1$--parameter family of quadratic differentials $q_t$ from $q$ 
so that, for $t\in \mathbb R$, if $\zeta=\xi+i\eta$ are natural coordinates for $q$, 
then $\zeta_t=e^{t} \xi+i e^{-t} \eta$ are natural coordinates for $q_t$. 
Let  $X_t$ be the conformal structure associated to $q_t$. Then 
$\calG:\mathbb R\to \calT(S)$ which sends $t$ to $X_t$, is a \Teich geodesic.

As mentioned above, a quadratic differential $q$ on a Riemann surface $X$ 
defines a  pair or transverse measured foliations. In fact, by a theorem of J. Hubbard 
and H. Masur \cite{masur:QDF}, $q$ is defined uniquely by $X$ and its vertical
measured foliation. More precisely, for any $X\in \calT(S)$ and
a measured foliation $\bnu$ on $S$, there is a unique quadratic differential $q$
on $X$ whose vertical foliation is measure equivalent to (a positive real multiple of) $\bnu$.

\subsection{Curves and markings}
A simple closed curve on $S$ is \textit{essential} if it does not bound a disk 
or a punctured disk on $S$. The free homotopy class of an essential 
simple closed curve will be called a \textit{curve}. Given
two  curves $\alpha$ and $\beta$, we denote by $\I(\alpha,\beta)$ the minimal
intersection number between the representatives of $\alpha$ and $\beta$. 
Let $A$ be a collection of curves in $S$. We say that $A$ is \textit{filling} 
 if for any curve $\beta$ in $S$,  $\I(\alpha,\beta)>0$ for some $\alpha\in A$.
 
A \textit{pants decomposition} of $S$ is a set $\calP$ consisting of pairwise 
disjoint curves. A \textit{marking} of $S$ is a pants decomposition $\calP$ together 
with a set $\calQ$ of transverse curves defined in a following way: for each
$\alpha\in \calP$ there is exactly one curve $\beta \in \calQ$ that intersects  $\alpha$ minimally and is disjoint from any other curve in $\calP$.  We will refer to the marking by its underlying set of curves $\mu = \calP\cup\calQ$. The intersection number of a marking $\mu$ with a curve $\alpha$ is defined to be the sum $\I(\mu,\alpha) = \sum_{\gamma \in \mu} \I(\gamma,\alpha)$.

\subsection{Curve graph}   
The \emph{curve graph} $\calC(S)$ of $S$ is a graph whose vertex set
$\calC_0(S)$ is the set of curves on $S$, and where two vertices are connected by 
an edge if the corresponding curves are disjoint, that is, have zero intersection number. 
This graph is the $1$--skeleton of the curve complex introduced by Harvey
\cite{harvey:BS}. Let $d_S$ be the path metric on $\calC(S)$ 
with unit length edges.  
\begin{theorem}[\cite{minsky:CCI}]
There exists $\delta\geq 0$ such that $\calC(S)$ is $\delta$-- hyperbolic.
\end{theorem}

\subsection{Relative twisting}
Let $\alpha$, $\beta$ and $\gamma$ be curves where $\alpha$ intersects both 
$\beta$ and $\gamma$. We define the relative twisting of $\beta$ and $\gamma$ 
around $\alpha$, denoted by  $\twist_\alpha(\beta, \gamma)$, as follows. 
Let $S_\alpha$ be the annular cover of $S$ associated to $\alpha$.
The $S_\alpha$ has a well defined boundary at infinity and can be considered
as a compact annulus. Let $\tilde \beta$ and $\tilde \gamma$
be lifts of $\beta$ and $\gamma$ to $S_\alpha$, respectively, that
are not boundary parallel. Assume $\beta$ and $\gamma$ intersect minimally 
in their free homotopy class. Then 
\[
\twist_\alpha(\beta, \gamma) = \I(\tilde \beta, \tilde \gamma).
\]
We can also define $\twist_\alpha(X, \beta)$ for a point $X \in \T(S)$. 
Consider the annular cover  $X_\alpha$ of $X$ and let $\tau_X$
be the arc in $X_\alpha$ that is perpendicular to $\alpha$. We define
\[
\twist_\alpha(X, \beta) = \twist_\alpha(\tau_X, \beta).
\]
 
The following is a straight-forward consequence of the Bounded Geodesic Image Theorem \cite[Theorem 3.1]{minsky:CCII}.

\begin{theorem} \label{Thm:BGIT}
Given constants $\kappa, \lambda > 0$, there exists a constant $\K=\K(S)$ such 
that the following holds. Let $\{\gamma_i\}$ be the vertex set of a $1$--Lipschitz $(\kappa,\lambda)$--quasi-geodesic
in $\calC(S)$ and let $\alpha$ be a curve on $S$. If every $\gamma_i$ intersects 
$\alpha$, then for every $i$ and $j$, 
\[
\twist_\alpha (\gamma_i, \gamma_j)\leq \K.
\]
\end{theorem}
\subsection{Geodesic laminations and foliations}
Let $X$ be a hyperbolic metric on $S$. A \emph{geodesic lamination} $\lambda$ is a
  closed subset of  $S$ which is a union of disjoint simple complete
  geodesics in the metric of $X$. Given another hyperbolic metric on $S$, 
  there is a canonical one-to-one
  correspondence between the two spaces of geodesic laminations.

A \emph{transverse measure} on a geodesic lamination $\lambda$  is an assignment of a  
Radon measure on each arc transverse to the leaves of the lamination, supported on the intersection, which is invariant under isotopies preserving the transverse intersections with the leaves of $\lambda$.  A {\em measured lamination} $\bar\lambda$ is a geodesic lamination $\lambda$ together with a transverse measure (which by an abuse of notation we also denote $\bar \lambda$).
The space of measured geodesic laminations on $S$ admits a natural topology, 
and the resulting space is denoted by $\calM\calL(S)$.  The set of non-zero 
laminations up to scale, the projective measured laminations,
is denoted $\calP\calM\calL(S)$.
The set $\calC_0(S)$ can be identified with a subset of $\calM\calL(S)$ by taking
geodesic representatives with transverse counting measure. The
intersection number $\I(\cdot\,,\cdot)\from\calC_0(S)\times\calC_0(S)\to \RR_+$
extends naturally to a continuous function 
$$
\I(\cdot\,,\cdot)\from\calM\calL\times\calM\calL\to \RR_+.
$$ 
When $\bar\lambda$ is a measured lamination and $\alpha$ a curve, 
$\I(\bar\lambda,\alpha)$ is the total mass of transverse measure on the geodesic representative of $\alpha$
assigned by $\bar\lambda$; see  \cite{thurston:GT} and \cite{penner:TT} for details.
 
A measured lamination $\bar\lambda$ is \emph{filling} if  for any curve
$\alpha$, $\I(\bar\lambda,\alpha)>0$; it is called \emph{minimal} if 
each of its leaves is a dense subset of $\lambda$.
  
There is a forgetful map from $\PML(S)$ to the set of \emph{unmeasured lamination}, where the points are {\em measurable geodesic laminations}---laminations admitting some transverse measure---and the map $\bar\lambda \mapsto \lambda$ simply forgets the transverse measure.
We consider the subset of minimal, filling, unmeasured laminations, and give it the quotient topology of the subspace topology on $\PML(S)$.  The resulting space is denoted $\EL(S)$ and is called the space of {\em ending laminations}.
  
\begin{theorem}[\cite{klarreich:bc}, Theorem 1.3; \cite{hamenstad:GB}]\label{KL} 
The Gromov boundary at infinity of the curve complex $\calC(S)$ is the 
space of minimal filling laminations on $S$. If a sequence of curve $\{\alpha_i\}$ 
is a quasi-geodesic in $\calC(S)$ then any limit point of $\alpha_i$ in $\PML(S)$ projects to $\nu$ under the forgetful map.  
\end{theorem}

\subsection{Train-tracks}
We refer the reader to \cite{thurston:GT} and \cite{penner:TT}
for a detailed discussion of train tracks and simply recall the relevant notions for our purposes.   A {\em train-track} is a smooth embedding 
of a graph with some additional structure.  Edges of the graph are called
{\em branches} and the vertices are called {\em switches}. An {\em admissible weight} on a 
train-track is an assignment of positive real numbers to the branches that satisfy
the {\em switch conditions}, that is, at each switch the sum of incoming weights is equal to the sum of outgoing weights.

A measured lamination $\bar\lambda \in \ML(S)$ is {\em carried} by a train track $\tau$ if there exists a map $f \colon S \to S$ homotopic to the identity whose restriction to $\lambda$ is an immersion, and such that $f(\lambda) \subset \tau$.  In this case $\bar\lambda$ determines an admissible weight on $S$, and this weight uniquely determines $\bar\lambda \in \ML(S)$.
A simple closed curve $\gamma$ meets $\tau$ {\em efficiently} if it intersects $\tau$ transversely and if there are no {\em bigons} formed by $\gamma$ and $\tau$.  If $\bar\lambda$ is carried by $\tau$ and $\gamma$ meets $\tau$ efficiently, then $i(\bar\lambda,\gamma)$ is calculated as the dot product of the weight vector of $\bar\lambda$ and the vector recording the number of intersection points of $\gamma$ with each of the branches.


\subsection{$\calM\calL(S)$ versus $\calM\calF(S)$} \label{S:MLMF}
There is a natural homeomorphism between the space of measured geodesic laminations and singular measured foliations on $S$ that is the
``identity'' on the respective inclusions of the sets of weighted simple closed curves.
More precisely, given a  measured foliation $F$ on a hyperbolic surface $X$, there is a measured 
geodesic lamination $\bar\lambda_F$, obtained from $F$ by straightening the leaves 
of $F$. This assigment is 1-1, onto and for every curve $\gamma$,
we have  $\I(F,\gamma)=\I(\bar\lambda_F,\gamma)$.  See for \cite{levitt:FL} for details.

For a quadratic differential $(X,q)$, denote the horizontal and the vertical 
measured foliations by $F_-$ and $F_+$. Then, for every curve $\gamma$, 
we define the \emph{horizontal length} of $\gamma$ at $(X,q)$ to be
\[
h_q(\gamma)= \I(\gamma, F_+)
\]
and the \emph{vertical length} of $\gamma$ at $(X,q)$ to be
\[
v_q(\gamma)= \I(\gamma, F_-).
\]

\section{Minimal lamination} \label{Sec:Construction}

\subsection{Construction of curves}
Here and throughout, we consider the five-times punctured sphere $S$ as being obtained from a pentagon in the plane, doubled along its boundary.  Our pictures illustrate $S$ by showing one of these pentagons.
Given this, we consider the four curves $\gamma_0,\gamma_1,\gamma_2,\gamma_3$ on 
$S$ shown in \figref{fivecurves}.  Let $\mu$ 
be the marking $\mu = \{ \gamma_0, \gamma_1, \gamma_2, \gamma_3\}$.

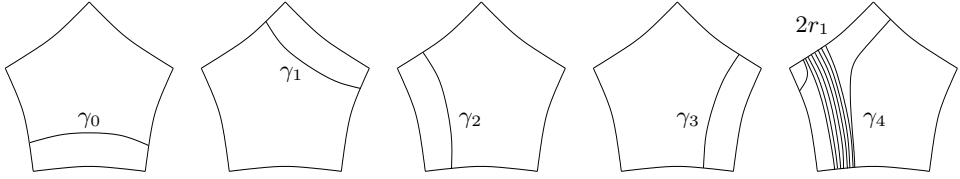
\begin{figure}[htb]
\labellist
\small\hair 2pt
 \pinlabel {$\gamma_0$} [ ] at 55  32
 \pinlabel {$\gamma_1$} [ ] at 185 60
 \pinlabel {$\gamma_2$} [ ] at 300 32
 \pinlabel {$\gamma_3$} [ ] at 440 32
 \pinlabel {$\gamma_4$} [ ] at 560 32 
 \pinlabel {$2r_1$} [ ] at 520 90
\endlabellist
\begin{center}
\includegraphics[width=5in,height=.9in]{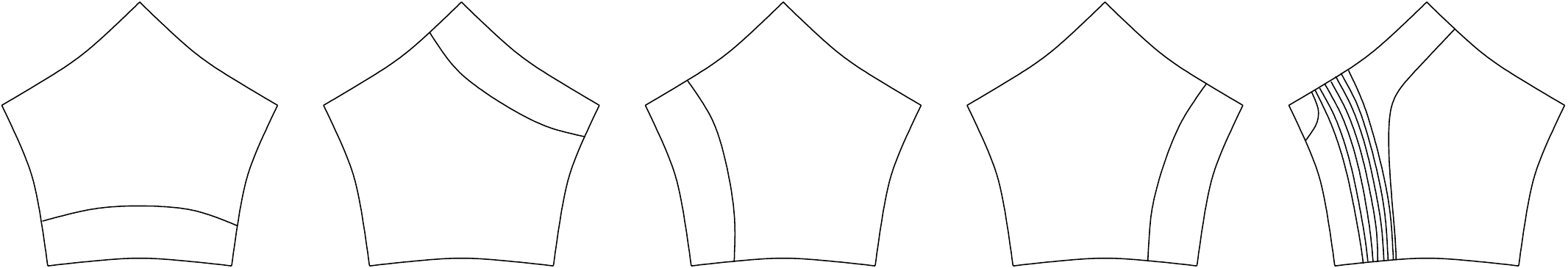}
\caption{The curves $\gamma_0,\ldots,\gamma_4$ in $S$.  By construction $\twist_{\gamma_2}(\gamma_0,\gamma_4) = r_1$.  The curves 
$\gamma_{i-2},\gamma_{i-1},\gamma_{i},\gamma_{i+1},\gamma_{i+2}$ differ from 
those shown here by replacing $r_1$ with $r_{i-1}$ and applying $\Phi_{i-2}$ to 
all five curves simultaneously (see \lemref{twist}).}
\label{Fig:fivecurves}
\end{center}
\end{figure}

Let $\rho$ denote the finite order symmetry of $S$ which is realized as the obvious 
counter-clockwise rotation by an angle $4\pi/5$ and let 
$D = D_{\gamma_2}$ be the left Dehn twist about $\gamma_2$.   
For any $r \in \ZZ$, let $\phi_r = D^r \circ \rho$.

Fix a sequence $\{r_i\}_{i = 1}^\infty \subset \ZZ_+$ of positive integers. Set
\[ \Phi_i = \phi_{r_1} \phi_{r_2} \cdots \phi_{r_{i-1}} \phi_{r_i} \]
and set $\Phi_0 = id$. Observe that, for $i = 1,2,3$ we have 
$\Phi_i(\gamma_0)  = \gamma_i$.  In fact, for any $r,s,t$, we have
\[ 
\gamma_1 = \phi_r(\gamma_0), \quad \gamma_2 = 
\phi_s \phi_r(\gamma_0), \quad \gamma_3 = \phi_t \phi_s \phi_r(\gamma_0).
\]

Now define $\gamma_i = \Phi_i(\gamma_0)$, for every $i \in \ZZ_+$.  
It follows that
\begin{equation} \label{Eq:altdefs}
\gamma_i = \Phi_i(\gamma_0) = \Phi_{i-1}(\gamma_1) = 
\Phi_{i-2}(\gamma_2) = \Phi_{i-3}(\gamma_3)
\qquad \forall i \in \ZZ. 
\end{equation}

\begin{lemma} \label{Lem:twist} 
For all $i \geq 2$, we have
\[ \twist_{\gamma_i}(\gamma_{i-2},\gamma_{i+2}) = r_{i-1}.\]
\end{lemma}
\begin{proof}
This follows from \eqref{Eq:altdefs}:
\begin{align*}
	\twist_{\gamma_i}(\gamma_{i-2},\gamma_{i+2}) & = \twist_{\Phi_{i-2}(\gamma_2)}(\Phi_{i-2}(\gamma_0),\Phi_{i-2} \phi_{r_{i-1}}(\gamma_3))\\
	& = \twist_{\gamma_2}(\gamma_0,\phi_{r_{i-1}}(\gamma_3))\\
	& = r_{i-1}. \qedhere
\end{align*} 
\end{proof}

\subsection{Curve complex quasi-geodesic}
In the construction above any two consecutive curves $\gamma_i$ and $\gamma_{i+1}$ 
are disjoint, and we can 
naturally view these as vertices of an infinite edge path in the curve complex 
$\calC(S)$, which we simply denote by $\{\gamma_i\}$.

We next show that the path $\{\gamma_i\}_i$ is a quasi-geodesic in
the curve complex and therefore  limits to a minimal filling lamination 
in the boundary. 

\begin{lemma} \label{Lem:Quasi-Geodesic} There exists $R>0$ such that,
if for some $i_0\geq 1$, 
\[
r_i>R  \quad \forall i \geq i_0,
\]
then the path $\{\gamma_i\}$ is an (infinite diameter) quasi-geodesic in $\calC(S)$. 
\end{lemma}
\begin{proof}
Let  $\K$ be the constant associated to the bounded geodesic 
image theorem (\thmref{BGIT}).  Pick $R > \K$ (the exact value of $R$
to be determined below). Set $J = J(R) = (R-\K)/2$ and choose $i_0 = i_0(R) \geq 0$ 
so that for all $i \geq i_0$ we have $r_i \geq R$. 

\begin{claim}
For every $i_0 \leq j < h < k$ with 
\begin{equation} \label{Eq:qgeodesic1} 
k-j \leq J, \quad k-h \geq 2 \quad\text{and}\quad h-j \geq 2,
\end{equation}
the curves $\gamma_j,\gamma_k$ fill $S$ and
\begin{equation} \label{Eq:qgeodesic2}
\twist_{\gamma_h}(\gamma_j,\gamma_k) \geq R - 2(k-j).
\end{equation}
\end{claim}
\begin{proof}[Proof of Claim]
The proof is by induction on $n=k-j$.  The case $n=4$ follows from \lemref{twist} 
and the fact that $r_{h-1} \geq R \geq R - 2(k-j)$ for all $h \geq i_0$.
We assume that the conclusion of the claim holds for all $j < h < k$ 
satisfying \eqref{Eq:qgeodesic1} and $k-j \leq n < J$, where $n \geq 4$.  

Now suppose $j \leq h \leq k$ satisfy \eqref{Eq:qgeodesic1} and $k-j = n+1 \leq J$.
Since $n + 1 > 4$, either $k-h \geq 3$ or $h-j \geq 3$.  Without loss of 
generality, assume $k-h \geq 3$ (a similar argument handles the other case). 
By the induction hypothesis, \eqref{Eq:qgeodesic2} holds for $j < h < k-1$:
\[ \twist_{\gamma_h}(\gamma_j,\gamma_{k-1}) \geq R - 2(k-1-j).\]
Note that since $2 \leq k-h \leq n$ we have $\I(\gamma_h,\gamma_k) \neq 0$: this follows from \figref{fivecurves} when $k-h=2$ or $3$, and is a consequence of the inductive assumption when $k-h \geq 4$ (to see this, replace $h$ with $j$).
Consequently, the projection of $\gamma_k$ to the arc complex $\calC(\gamma_h)$ of the annular neighborhood of $\gamma_h$ is nonempty.  In particular, the triangle inequality in $\calC(\gamma_h)$, together with the fact that the projection to $\calC(\gamma_h)$ is $2$--Lipschitz, implies
\begin{align*}
	\twist_{\gamma_h}(\gamma_j,\gamma_k) & \geq \twist_{\gamma_h}(\gamma_j,\gamma_{k-1}) - \twist_{\gamma_h}(\gamma_{k-1},\gamma_k)\\
	& \geq R - 2(k-1-j) - 2 = R - 2(k-j).
\end{align*}
Therefore \eqref{Eq:qgeodesic2} holds for $j < h < k$.

To prove that the curves $\gamma_j,\gamma_k$ fill $S$, suppose this is not the case.  Since $\gamma_j$ and $\gamma_{k-1}$ do fill, their distance $d_S(\gamma_j,\gamma_{k-1})$ in $\calC(S)$ is at least $3$.  
By the triangle inequality $d_S(\gamma_j,\gamma_k) = 2$.  

Choose any $h$ with $j < h < h+1 < k$ so that $k-h-1 \geq 2$ and $h-j \geq 2$.
Since
\[ 
\twist_{\gamma_h}(\gamma_j,\gamma_k),
\twist_{\gamma_{h+1}}(\gamma_j,\gamma_k) 
\geq R - 2(k-j) \geq R - 2J = \K 
\]
it follows from \thmref{BGIT} that a geodesic $\gamma_j,\gamma',\gamma_k$ in the curve complex must pass through a curve $\gamma'$ having zero intersection number with both $\gamma_h$ and $\gamma_{h+1}$.    The only such curves are $\gamma_h$ and $\gamma_{h+1}$.   This is impossible since $\gamma_j$ and $\gamma_k$ both intersect $\gamma_h$ and $\gamma_{h+1}$.  This contradiction completes the proof of the claim.
\end{proof}

Next we prove
\begin{claim}
For every $J_0 \leq j \leq k$ so that $k-j \leq J$, we have
\[ d_S(\gamma_j,\gamma_k) \geq (k-j)/3. \]
\end{claim}
\begin{proof}[Proof of Claim]
By \thmref{BGIT}, we know that for any $h$ with $j < h < k$ satisfying 
\eqref{Eq:qgeodesic1}, any geodesic $\calG$ from $\gamma_j$ to $\gamma_k$ 
must contain a vertex $\gamma_h'$ with $i(\gamma_h,\gamma_h') = 0$.   By the 
first claim, if $j < h < f < k$ with $f-h \geq 4$, then $\gamma_h$ and $\gamma_f$ fill 
$S$, so there is no curve disjoint from both of them.  It follows that any vertex in 
$\calG$ has zero intersection number with at most three vertices of $\{ \gamma_i \}$, 
and these vertices are consecutive.  Furthermore, the only curve disjoint from three 
consecutive vertices of $\{\gamma_i\}$ is the middle curve of the three curves.

Consider the $k-j-3$ vertices $\gamma_{j+2},\ldots,\gamma_{k-2}$ of $\{ \gamma_i\}$.  
For each $h \in \{j+2,j+3,\ldots,k-2\}$, let $\gamma_h'$ denote (one of the) vertices of 
$\calG$ having zero intersection number with $\gamma_h$.  By the previous paragraph, 
this is at most three-to-one, hence the number of vertices in $\calG$ is at least 
$2 + (k-j-3)/3$, and so the distance between $\gamma_j$ and $\gamma_k$ is bounded by
\[
d_S(\gamma_j,\gamma_k) \geq 1 + (k-j)/3  - 1 = (k-j)/3
\]
as required.
\end{proof}
\begin{remark} With a little more work, the $1/3$ can be replaced by $2/3$.
\end{remark}

We have shown that $\{\gamma_i\}_{i \geq i_0}$ is a $J$--local $(3,0)$--quasi-geodesic.  
Since $\calC(S)$ is $\delta$--hyperbolic, by taking $R$, and hence $J = J(R)$, sufficiently 
large, it follows that this path is a quasi-geodesic (see for example \cite{bridson:NPC}, 
Theorem III.1.13).
\end{proof}

Combining this with Theorem~\ref{KL} immediately implies
\begin{corollary} \label{C:GromovBoundaryLimit}
With the assumptions as in Lemma~\ref{Lem:Quasi-Geodesic} there is an ending lamination $\nu$ on $S$ representing a point on the Gromov boundary of $\calC(S)$ such that
\[ \lim_{k \to \infty} \gamma_k = \nu.\]
Consequently every limit point in $\PML(S)$ of $\{\gamma_k\}$ defines a projective class of transverse measure on $\nu$.
\end{corollary}
We emphasize that $\nu$ is unmeasured geodesic lamination, that is, it admits a transverse 
measure but it is not equipped with one.

\section{Computing intersection numbers and convergence}\label{Sec:Intersections}

For the remainder of the paper, we will assume $0 < \epsilon < \frac{1}{2}$ and that the sequence of integers $\{r_i\}$ satisfies 
\begin{equation} \label{E:assumptions-on-r_i}
\frac{1}{\epsilon} \leq r_1  \text{  and  } r_i \leq \epsilon r_{i+1}.
\end{equation}
Consequently, $\frac{1}{r_i} \leq \epsilon^i < \frac{1}{2^i}$.
By \lemref{Quasi-Geodesic}, this assumption guarantees that 
 the sequence $\{\gamma_i\}$ is a quasi-geodesic.
In this section we will provide estimates on the intersection numbers 
between pairs of curves in  $\{\gamma_i\}$, up to uniform multiplicative errors.  
Specifically, we prove the following.

\begin{proposition} \label{Prop:intersectiongamma}
For all $j < k$ with $j \equiv k \mod 2$ we have
\begin{equation} \label{Eq:k}
\I(\gamma_{j-1},\gamma_k) \emul \I(\gamma_j,\gamma_k) \emul  
  \prod_{\tiny \begin{array}{c}i = j \\ i \equiv j \mod 2 \end{array} }^{k-4} 2 r_{i+1}.
\end{equation}
\end{proposition}

To prove this proposition we will analyze a certain train track that carries all but the first two curves in our sequence.  Specifically,
\figref{traintrack} shows a train track $\tau$ which is $\phi_r$--invariant, for all $r \in \ZZ_+$.  Any nonnegative vector
\[ {\bf s} = \left( \begin{array}{c}  s_1\\ s_2 \\ s_3 \\ s_4 \\ s_5 \end{array} \right)  \]
with $s_1 + s_4 + s_5 = s_2 + s_3$ and $s_1+s_2+s_3+s_4+s_5 >0$ determines a measure on $\tau$ by assigning weights to the branches as shown in \figref{traintrack}.
Setting
\[ 
{\bf v}_2 = \left( \begin{array}{c} 1 \\ 1 \\ 0 \\ 0 \\ 0 \end{array} \right) 
\qquad \text{and}\qquad
{\bf v}_3 = \left( \begin{array}{c} 0 \\ 0 \\ 1 \\ 1 \\ 0 \end{array} \right) 
\]
we can check that ${\bf v}_2$ defines the curve $\gamma_2$ and ${\bf v}_3$ 
defines the curve $\gamma_3$.

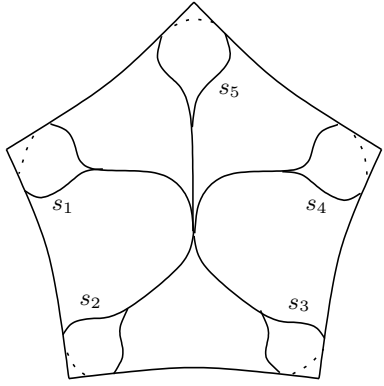
\begin{figure}[htb]
\labellist
\small\hair 2pt
 \pinlabel {$s_1$} [ ] at 17 48
 \pinlabel {$s_2$} [ ] at 25 22
 \pinlabel {$s_3$} [ ] at 85 21
 \pinlabel {$s_4$} [ ] at 90 48
 \pinlabel {$s_5$} [ ] at 65 80 
\endlabellist
\begin{center}
\includegraphics[width=2in,height=2in]{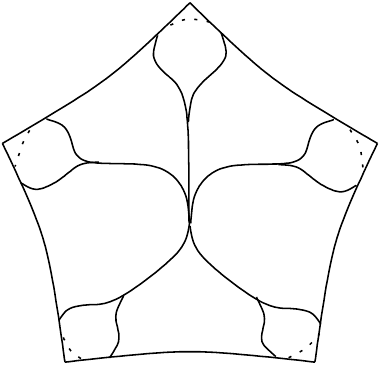}
\caption{The train track $\tau$.}
\label{Fig:traintrack}
\end{center}
\end{figure}

By a direct calculation we see that the action of $\phi_r$ on the weights on the train track (recorded as a column vector) is given by left multiplication by the matrix
\[ M_r = \left( 
	\begin{array}{ccccc} 
	0 & 0 & 0 & 2r -1 & 2r\\
	0 & 0 & 0 & 2r & 2r+1\\
	1 & 0 & 0 & 0 & 0\\
	0 & 1 & 0 & 0 & 0\\
	0 & 0 & 1 & 0 & 0
	\end{array}\right).
\]
In particular, by \eqref{Eq:altdefs} we have
\[ M_{r_1} \cdots M_{r_{i-2}} {\bf v}_2 = M_{r_1}M_{r_2} \cdots M_{r_{i-3}} {\bf v_3} \]
defines the curve $\gamma_i$.

The entries of the products $M_{r_1} \cdots M_{r_j}$ blow up as $j \to \infty$, but we can calculate the rate at which they blow up.  Using this we will scale so that certain subsequences of these scaled matrices converge.  As we explain below, the estimates we obtain can also be used to appropriately estimate intersection numbers.  

We now define
\[ N_i = \frac{1}{2r_{i+1}} M_{r_i}M_{r_{i+1}} 
	= \left( 
	\begin{array}{ccccc} 
	0 & \frac{2r_i - 1}{2r_{i+1}} & \frac{r_i}{r_{i+1}} & 0 & 0\\
	0 & \frac{r_i}{r_{i+1}} & \frac{2r_i + 1}{2r_{i+1}} & 0& 0\\
	0 & 0 & 0 & \frac{2r_{i+1}-1}{2r_{i+1}} & 1\\
	0 & 0 & 0 & 1 & \frac{2r_{i+1}+1}{2r_{i+1}}\\
	\frac{1}{2r_{i+1}} & 0 & 0 & 0 & 0
	\end{array}\right)
\]
We will analyze the matrices one row at a time, and will apply an inductive argument.  To facilitate this, let ${\bf x} = (x_1 , x_2 , x_3 , x_4 , x_5 )$ be a nonnegative row vector and $\|{\bf x}\|_\infty = \max\{ x_1,x_2,x_3,x_4,x_5 \}$.  For any $i > 0$ we can calculate 
\begin{equation} \label{Eqn:xformula} {\bf x}N_i \!=\! \left(  \! \tfrac{x_5}{2 r_{i+1}} ,
\tfrac{(2r_i-1)x_1+2r_ix_2}{2r_{i+1}} ,
\tfrac{2r_i x_1 + (2r_i+1)x_2}{2r_{i+1}} , 
\tfrac{(2 r_{i+1} - 1)x_3}{2 r_{i+1}} \!+\! x_4 , x_3 \!+\! \tfrac{(2 r_{i+1} + 1)x_4}{2 r_{i+1}}   \! \right) \end{equation}
From this and the assumption \eqref{E:assumptions-on-r_i} on $\{r_i\}$, we easily obtain the following
\begin{equation} \label{Eqn:basicxbound} {\bf x} N_i \leq \left(  \epsilon^{i+1} x_5 , \epsilon(x_1 + x_2)  , \epsilon(x_1 + x_2) + \epsilon^{i+1}  x_2  , x_3+ x_4 , (x_3+ x_4)(1+\epsilon^{i+1}) \right). \end{equation}
Here and in all that follows, an inequality of matrices or vectors means that this inequality holds for every entry.

The most general estimates we will need are then provided by the following.

\begin{proposition} \label{P:matrixproductbounds}
Fix any integer $J >0$ such that
\[ \delta = 2 \epsilon + \epsilon^J < 1.\]
Then for any nonnegative ${\bf x} = (x_1 , x_2 , x_3 , x_4 , x_5 )$ and all $j \geq J$ and $\ell \geq 0$, we have
\[ \left( 0, 0, 0, x_4, x_4 \right) \leq {\bf x} N_j N_{j+2} \ldots N_{j+ 2\ell} \hspace{16cm} \]
\[ \hspace{1cm} \leq \|{\bf x}\|_\infty \left( (2 \epsilon)^{\ell+1} , (2 \epsilon)^{\ell+1} , \delta (2 \epsilon)^\ell , 2 + \delta \sum_{i=0}^{\ell-1} (2 \epsilon)^i , \left( 2 + \delta \sum_{i=0}^{\ell-1} (2 \epsilon)^i  \right) (1 + \epsilon^j) \right). \]
Furthermore, the first inequality is valid for all $j \geq 0$ and $\ell \geq 0$.
\end{proposition}

\begin{proof}  For every $j,\ell$, the lower bound on ${\bf x} N_j N_{j+2} \ldots N_{j+ 2\ell}$ claimed in the statement is immediate from Equation (\ref{Eqn:xformula}), so we concentrate on the upper bound.

Fix $j \geq J$, set ${\bf 1} = (1 , 1 , 1 , 1 , 1)$.  Define ${\bf y}^{(0)} = {\bf 1}N_j$ and for $\ell \geq 1$ set
\[ {\bf y}^{(\ell)} = {\bf y}^{(\ell-1)} N_j N_{j+2} \ldots N_{j+ 2\ell}.\]
Since ${\bf x} \leq \|{\bf x}\|_\infty {\bf 1}$, the upper bound will follow if we can show that for all $\ell \geq 0$
\begin{equation} \label{Eqn:xrecursive} {\bf y}^{(\ell)} \leq \left( (2 \epsilon)^{\ell+1} , (2 \epsilon)^{\ell+1} , \delta (2 \epsilon)^\ell , 2 + \delta \sum_{i=0}^{\ell-1} (2 \epsilon)^i , \left( 2 + \delta \sum_{i=0}^{\ell-1} (2 \epsilon)^i  \right) (1 + \epsilon^{j+2\ell+1}) \right)\end{equation}
We prove this by induction on $\ell$.  Appealing to (\ref{Eqn:basicxbound}), we can easily verify that the inequality is valid for $\ell = 0$.  Assuming it holds for some $\ell \geq 0$, we prove it for $\ell+1$.  

To prove this we multiply by $N_{j+2\ell+2}$ on both sides of the inequality (\ref{Eqn:xrecursive}):
\[  {\bf y}^{(\ell)} N_{j+2 \ell + 2} \leq \hspace{10cm} \]
\[  \left( (2 \epsilon)^{\ell+1} , (2 \epsilon)^{\ell+1} , \delta (2 \epsilon)^\ell , 2 + \delta \sum_{i=0}^{\ell-1} (2 \epsilon)^i , \left( 2 + \delta \sum_{i=0}^{\ell-1} (2 \epsilon)^i  \right) (1 + \epsilon^{j+2\ell+1}) \right)N_{j+2\ell + 2}. \]
The left-hand side is precisely ${\bf y}^{(\ell+1)}$.  We will bound the right-hand side by appealing to (\ref{Eqn:basicxbound}) with ${\bf x}$ replaced by the row vector on the right-hand side of (\ref{Eqn:xrecursive}).  We carry this out for each entry individually.

For the first entry we note that $\delta \displaystyle{\sum_{i=0}^{\ell-1} (2 \epsilon)^i  \leq \delta \ell < \ell}$ and so we have
\[ \begin{array}{rcl} y^{(\ell + 1)}_1 & \leq & \epsilon^{j + 2(\ell+1) + 1} \left( (2+\ell)(1+\epsilon^{j + 2\ell + 1}) \right)\\\\
& = & \epsilon^{j+\ell+1}(1+\epsilon^{j+2\ell+1})(2+\ell)\epsilon^{\ell+2} < (2 \epsilon)^{\ell+2}.
\end{array} \]
In the last inequality we have used the fact that $2+\ell \leq 2^{\ell+1}$ for all $\ell \geq 0$ and that $\epsilon^{j+\ell+1}(1+\epsilon^{j+2\ell+1}) < 1$ since $\epsilon < \tfrac12$.

We similarly bound the other entries:
\[ \begin{array}{rcl} 
y^{(\ell+1)}_2 & \leq &  \epsilon((2 \epsilon)^{\ell+1} + (2 \epsilon)^{\ell+1}) \leq 2\epsilon(2 \epsilon)^{\ell+1} = (2 \epsilon)^{\ell+2}\\\\
y^{(\ell+1)}_3 & \leq & \epsilon((2 \epsilon)^{\ell+1} + (2 \epsilon)^{\ell+1}) + (2 \epsilon)^{\ell+1} \epsilon^{j + 2\ell + 3} \\\\
& = & (2 \epsilon)^{\ell+1}(2 \epsilon + \epsilon^{j+ 2\ell + 3}) < (2 \epsilon)^{\ell+1}(2 \epsilon + \epsilon^j) \leq \delta (2 \epsilon)^{\ell+1} \\\\
y^{(\ell+1)}_4 & \leq & \displaystyle{\delta (2 \epsilon)^\ell + 2 + \delta \sum_{i=0}^{\ell-1} (2 \epsilon)^i = 2 + \delta \sum_{i=0}^\ell (2 \epsilon)^i }\\\\
y^{(\ell+1)}_5 & \leq & \displaystyle{\left(\delta (2 \epsilon)^\ell + 2 + \delta \sum_{i=0}^{\ell-1} (2 \epsilon)^i)\right)\left(1 + \epsilon^{j + 2\ell + 3} \right)}\\
 & \leq & \displaystyle{\left( 2 + \delta \sum_{i=0}^\ell (2 \epsilon)^i \right)\left(1 + \epsilon^{j + 2\ell + 3} \right).}
\end{array} \]
This completes the proof of the proposition.
\end{proof}

We deduce two easy corollaries of Proposition~\ref{P:matrixproductbounds}.

\begin{corollary} \label{C:matrices converge}
For any $j \geq 0$, the sequence of matrices $\{ N_jN_{j+2} \cdots N_{j+2\ell}\}_{\ell = 0}^\infty$ converges.  
\end{corollary}
\begin{proof}
We first claim that it suffices to prove the corollary for $j \geq J$.  To see this, suppose it is true for all integers greater than $J$ and suppose $j < J$.  Then let $j_0 \geq J$ be such that $j \equiv j_0$ (mod 2) and observe that
\[ \{N_j N_{j+2} \cdots N_{j+2 \ell}\}_{\ell = \tfrac{j_0-j}2}^\infty = N_jN_{j+2} \cdots N_{j_0-2} \{ N_{j_0} N_{j_0+2} \cdots N_{j_0 + 2 \ell}\}_{\ell = 0}^\infty. \]
By assumption $\{ N_{j_0} N_{j_0+2} \cdots N_{j_0 + 2 \ell}\}_{\ell = 0}^\infty$ converges.  Since matrix multiplication is continuous, the sequence $\{N_j N_{j+2} \cdots N_{j+2 \ell}\}_{\ell = 0}^\infty$ also converges, verifying the claim.

Now we suppose $j \geq J$.  Fix any integer $1 \leq s \leq 5$, and let $\{ {\bf x}^{(\ell)} \}_{\ell=0}^\infty$ be the sequence of $s^{th}$--row vectors of the products $\{N_j N_{j+2} \cdots N_{j+ 2 \ell}\}_{\ell=0}^\infty$.  Since
\[ {\bf x}^{(\ell+1)} = {\bf x}^{(\ell)}N_{j+2\ell + 2} = {\bf x}^{(0)} N_j N_{j+2} \cdots N_{j+2\ell+2}, \]
it follows from Equation (\ref{Eqn:xformula}) that the sequence of fourth entries $\{{\bf x}^{(\ell)}_4\}_{\ell=0}^\infty$ is increasing.  Furthermore, by Proposition~\ref{P:matrixproductbounds} this is bounded by the sum of a convergent geometric series (since $2\epsilon < 1$).  Consequently, $\{{\bf x}^{(\ell)}_4\}_{\ell=0}^\infty$ converges.  From Equation (\ref{Eqn:xformula}) we see that the ratio of the fourth and fifth entries tends to 1, and hence the sequence of fifth entries also converges.  Finally, by Proposition~\ref{P:matrixproductbounds}, the sequences of the first three entries all tend to zero.
Since $s$ was arbitrary, all rows converge, and hence the sequence of matrices converges.
\end{proof}

\begin{corollary} \label{C:matrix bounds intersection}
There exists a constant $C = C(\epsilon) > 0$ such that for any $j \geq 0$ and $\ell \geq 0$, we have
\[ \left( \begin{array}{ccccc} 0 & 0 & 0 & 0 & 0 \\ 0 & 0 & 0 & 0 & 0\\ 0 & 0 & 0 & \frac{1}{3} & 0 \\ 0 & 0 & 0 & 1 & 0 \\ 0 & 0 & 0 & 0 & 0 \\ \end{array} \right) \leq N_j N_{j+2} \cdots N_{j+2\ell} \leq \left( \begin{array}{ccccc} C & C & C & C & C \\  C & C & C & C & C \\ C & C & C & C & C \\ C & C & C & C & C \\ C & C & C & C & C \\ \end{array} \right). \]
\end{corollary}
\begin{proof}   The lower bound on the matrix product follows easily by taking ${\bf x}$ to be the $3^{rd}$ and $4^{th}$ rows of the identity matrix and then appealing to Equation~\ref{Eqn:xformula}, the lower bound in Proposition~\ref{P:matrixproductbounds}, and the assumption \ref{E:assumptions-on-r_i} on $\{r_i\}$.

Finding a constant $C > 0$ proving the upper bound in the corollary assuming $j \geq J$ will imply an upper bound for any $j$, with a potentially larger constant $C$, since there are only finitely many integers $0 \leq j < J$ (cf.~the proof of Corollary~\ref{C:matrices converge}).
In this case, let ${\bf x}$ be the $s^{th}$ row of the identity matrix, for any $s = 1,\ldots,5$.  Then $\|{\bf x}\|_\infty = 1$ and the right-hand side of Proposition~\ref{P:matrixproductbounds} bounds the entries of the $s^{th}$ row of $N_jN_{j+2} \cdots N_{j+2 \ell}$.  The first three entries are bounded by $2 \epsilon < 1$, while the last two are bounded by the sum of a convergent geometric series (independent of $j$ and $\ell$).  This bounds the entries of the $s^{th}$ row of $N_j N_{j+2} \cdots N_{j+2 \ell}$, independent of $j,\ell$, and $s$, completing the proof of the corollary.
\end{proof}

We now proceed to the
\begin{proof}[Proof of Propostion~\ref{Prop:intersectiongamma}.]
We first observe that \eqref{Eq:k} holds trivially if $k = j+2$. Indeed,  applying 
$\Phi_j^{-1}$ and comparing with  \figref{fivecurves}, we see that both intersection 
numbers are equal to $2$, while the right hand side (the empty product) is $1$.  
Therefore, we assume in what follows that $j \leq k - 4$ and $j \equiv k$ mod $2$.

The curves $\gamma_0$ and $\gamma_1$ are {\em not} carried by $\tau$, but 
do meet $\tau$ efficiently so that $\gamma_0$ meets the two large branches of 
$\tau$ on the bottom of \figref{traintrack} exactly once, while $\gamma_1$ 
meets the top and top right large branch each once. Setting
\[ 
{\bf u}_0 = \left( \begin{array}{c} 0 \\ 1 \\ 1 \\ 0 \\ 0 \end{array} \right)
\qquad \text{and} \qquad
{\bf u}_1 = \left( \begin{array}{c} 0 \\ 0 \\ 0 \\ 1 \\ 1 \end{array} \right) 
\]
it therefore follows that for $i \geq 4$ we have
\[ 
\I(\gamma_0,\gamma_i) = 
2 {\bf u}_0 \cdot (M_{r_1}M_{r_2} \cdots M_{r_{i-3}} {\bf v_3})
\]
and
\[ 
\I(\gamma_1,\gamma_i) = 
2 {\bf u}_1 \cdot (M_{r_1}M_{r_2} \cdots M_{r_{i-3}} {\bf v_3}). 
\]
More generally, for $k \geq j + 4$ we have
\begin{align*} 
	\I(\gamma_{j-1},\gamma_k) 
	& = \I(\Phi_{j-1}(\gamma_0),\Phi_{k-3}(\gamma_3))\\
	& = \I(\Phi_{j-1}(\gamma_0),\Phi_{j-1}\phi_{r_j} \cdots \phi_{r_{k-3}}(\gamma_3))\\
	& = \I(\gamma_0,\phi_{r_j} \cdots \phi_{r_{k-3}}(\gamma_3))\\
	& = 2 {\bf u}_0 \cdot (M_{r_j} \cdots M_{r_{k-3}} {\bf v_3})
\end{align*}
and similarly
\begin{align*}
	\I(\gamma_j,\gamma_k)
	& = \I(\gamma_1,\phi_{r_j} \cdots \phi_{k-3}(\gamma_3))\\
	& = 2 {\bf u}_1 \cdot (M_{r_j} \cdots M_{r_{k-3}}{\bf v_3}).
\end{align*}

On the other hand, because the entries of the vectors $\mathbf u_0,\mathbf u_1, \mathbf v_2,\mathbf v_3$ are all $0's$ and $1's$, the dot products in these equations are just the sums of certain entries of the matrix $M_{r_j} M_{r_{j+1}} \cdots M_{r_{k-3}}$.  Specifically, we have
\begin{equation} \label{Eq:entrysum1}
\I(\gamma_{j-1},\gamma_k) = 2 \sum_{s \in \{2,3\}, t \in \{3,4\} }  (M_{r_j} M_{r_{j+1}} \cdots M_{r_{k-3}})_{s,t}
\end{equation}
and
\begin{equation} \label{Eq:entrysum2}
\I(\gamma_j,\gamma_k) = 2 \sum_{s \in \{4,5\}, t \in \{3,4\} }  (M_{r_j} M_{r_{j+1}} \cdots M_{r_{k-3}})_{s,t}.
\end{equation}
Therefore, to prove \eqref{Eq:k}, Equations \eqref{Eq:entrysum1} and \eqref{Eq:entrysum2} show that it suffices to prove that for any $j$ and $k \equiv j$ (mod 2) with $j \leq k-4$
\[ \sum_{s \in \{2,3\} , t \in \{3,4\}}  (N_j N_{j+2}\cdots N_{k-4})_{s,t}
\, \, \emul \, \, 1 \, \, \emul \sum_{s \in \{4,5\} , t \in \{3,4\}}  (N_j N_{j+2}\cdots N_{k-4})_{s,t}  
\]
with implied multiplicative constants independent of $j$ and $k$.  This is immediate from Corollary~\ref{C:matrix bounds intersection}, with multiplicative constant $\max\{3,4C\}$.
\end{proof}

Applying both Corollaries~\ref{C:matrices converge} and \ref{C:matrix bounds intersection} we see that there are only two projective limits of $\{\gamma_k\}$.  More precisely, recall that 
$\mu =\{ \gamma_0, \gamma_1, \gamma_2, \gamma_3\}$.

\begin{proposition} \label{Prop:evenoddconverge}
Both of the sequences
\[ 
\left\{ \frac{\gamma_{2k}}{\I(\gamma_{2k},\mu)} \right\}_{k=0}^\infty 
\quad \text{and}\quad 
\left\{ \frac{\gamma_{2k+1}}{\I(\gamma_{2k+1},\mu)} \right\}_{k=0}^\infty 
\]
converge in $\ML(S)$ to non-zero measured laminations.
\end{proposition}
It is not difficult to see that up to subsequence 
$\{ \gamma_k/\I(\gamma_k,\mu)\}$ converges in $\ML(S)$.  However, the interesting 
point is that we only need two subsequences---the even and odd subsequences.
\begin{proof}
We give the proof for the odd sequence.  The proof that the even sequence 
converges is similar.
First, set
\[ c_{2k+1} = \prod_{\tiny \begin{array}{c}i = 1 \\ i \equiv 1 \mbox{ mod } 2 \end{array} }^{2k-3} 2 r_{i+1}. \]
Then we claim that 
\begin{equation} \label{Eq:ML-convergence} 
\left\{ \frac{\gamma_{2k+1}}{c_{2k+1}} \right\}_{k=0}^\infty
\end{equation}
converges to a non-zero element of $\ML(S)$.  From this we see that the limit
\[  
\lim_{k \to \infty} \frac{ \I(\gamma_{2k+1},\mu) }{c_{2k+1}} = 
\lim_{k \to \infty} \I\left( \frac{\gamma_{2k+1}}{c_{2k+1}}, \mu \right) 
\]
exists (and is nonzero since $\mu$ is a marking), hence the sequence
\[ 
\left\{\frac{\gamma_{2k+1}}{\I(\gamma_{2k+1},\mu)}\right\}_{k=0}^\infty
\]
converges as required.

To see that the sequence in \eqref{Eq:ML-convergence} converges we note 
that for all $k \geq 1$, $\gamma_{2k+1}/c_{2k+1}$ is carried by the train track $\tau$ and is 
defined by the vector
\[ 
\frac{1}{c_{2k+1}} M_{r_1}\cdots M_{r_{2k-2}}{\bf v_3} = 
N_1 N_3 \cdots N_{2k-3}{\bf v_3}. 
\]
Corollary~\ref{C:matrices converge} and continuity of matrix multiplication guarantees that the sequence of vectors $\{N_1N_3 \cdots N_{2k-3} {\bf v_3} \}_{k=2}^\infty$ converges.  Furthermore, by Corollary~\ref{C:matrix bounds intersection} the fourth entry of each of these vectors is at least $1$, and hence the limiting vector is nonzero.  The limiting vector defines a measured lamination carried by $\tau$ which is the desired limit of the sequence in (\ref{Eq:ML-convergence}).
\end{proof}

Not only is $\mu$ a marking, but any four consecutive curves in the sequence $\{\gamma_i\}$ form a marking 
of $S$.  Let 
\[
\mu_i = \Phi_i (\mu) = 
 \{ \gamma_i, \gamma_{i+1}, \gamma_{i+2}, \gamma_{i+3} \}.
\]  
We end this section with another application of \propref{intersectiongamma}.

\begin{corollary} \label{Cor:mark-int1} If $k \geq j$ and $k \equiv j$ mod $2$, then
\[ \I(\mu_{j-1},\gamma_k) \emul \I(\mu_j,\gamma_k) \emul  \prod_{\tiny \begin{array}{c}i = j \\ i \equiv j \mbox{ mod } 2 \end{array} }^{k-4} 2 r_{i+1} \]
\end{corollary}
\begin{proof}
Appealing to \propref{intersectiongamma} we have
\begin{align*}
	\I(\mu_{j-1},\gamma_k) 
	& = \I(\gamma_{j-1},\gamma_k) + \I(\gamma_j,\gamma_k) + \I(\gamma_{j+1},\gamma_k) + \I(\gamma_{j+2},\gamma_k)\\
	& \emul 2 \left( \prod_{\tiny \begin{array}{c}i = j \\ i \equiv j \mbox{ mod } 2 \end{array} }^{k-4} 2 r_{i+1} +
	\prod_{\tiny \begin{array}{c}i = j+2 \\ i \equiv j \mbox{ mod } 2 \end{array} }^{k-4} 2 r_{i+1} \right)\\
	& \emul  \prod_{\tiny \begin{array}{c}i = j \\ i \equiv j \mbox{ mod } 2 \end{array} }^{k-4} 2 r_{i+1}
\end{align*}
A similar computation proves the desired estimate for $\I(\mu_j,\gamma_k)$.
\end{proof}

\section{Non-unique ergodicity} \label{Sec:Ergodicity}

The parity constraints in \propref{intersectiongamma} and \corref{mark-int1} 
indicate an asymmetry that we wish to exploit.  To better emphasize this, we let 
$\{\alpha_i\}$ and $\{ \beta_i\}$ denote the even and odd subsequence of 
$\{ \gamma_i\}$, respectively.  More precisely, for all $i \geq 0$, set
\[ 
\alpha_i = \gamma_{2i} \qquad\text{and}\qquad \beta_i = \gamma_{2i+1}.
\]
We visualized the path $\{\gamma_i \}$ in the curve complex as 
\begin{equation*}
\xymatrix{
  \alpha_0 \ar[d] & \alpha_1 \ar[d]& \alpha_2 \ar[d]&  \text{\phantom{x}} \ldots \\
  \beta_0 \ar[ur] & \beta_1 \ar[ur]& \beta_2 \ar[ur]&   \text{\phantom{x}} \ldots      
}
\end{equation*}

Set $n_i = r_{2i-1}$ and $m_i = r_{2i}$ for all $i\geq 1$. Also set $m_0=1$. 
\begin{lemma} \label{Lem:twist2}
For all $i \geq 1$, we have
\[ 
\twist_{\alpha_i}(\alpha_{i-1},\alpha_{i+1}) = n_i
\qquad\text{and}\qquad 
\twist_{\alpha_i}(\mu,\nu) \eadd n_i.
\]
Similarly, 
\[ \twist_{\beta_i}(\beta_{i-1},\beta_{i+1}) = m_i
\qquad\text{and}\qquad 
\twist_{\beta_i}(\mu,\nu) \eadd m_i.
\]
\end{lemma}
\begin{proof}
The first equality follows from \lemref{twist}. The second inequality
follows from the first inequality, \lemref{Quasi-Geodesic} and
Theorem~\ref{Thm:BGIT}. 
\end{proof}

Likewise, \propref{intersectiongamma} and \corref{mark-int1} imply the following
\begin{corollary} \label{C:intersectionswithmu}
For all $i < k$ we have
\begin{equation} \label{Eq:n}
\I(\beta_{i-1}, \alpha_k) \emul \I(\alpha_i, \alpha_k) 
  \emul \prod_{j=i+1}^{k-1} 2 n_j,
\end{equation}  
and
\begin{equation} \label{Eq:m}
\I(\alpha_i, \beta_k) \emul  \I(\beta_i, \beta_k) 
  \emul \prod_{j=i+1}^{k-1} 2 m_j.
\end{equation}
Moreover, for all $i \geq 0$ we have
\begin{equation} \label{Eq:markn} 
\I(\mu,\alpha_{i+1})  \emul 
\prod_{j=1}^i 2 n_j
\end{equation}
and  
\begin{equation} \label{Eq:markm}
\I(\mu, \beta_{i+1})  \emul \prod_{j=1}^i 2 m_j.
\end{equation}
\end{corollary}


By \propref{evenoddconverge}, there are measured laminations $\bnu_\alpha,\bnu_\beta \in \ML(S)$ defined by the limits 
\begin{equation}\label{Eq:definition}
\bnu_\alpha = \lim_{i \to \infty} \frac {\alpha_i}{\I(\alpha_i, \mu)}
\qquad\text{and}\qquad
\bnu_\beta = \lim_{i \to \infty} \frac {\beta_i}{\I(\beta_i, \mu)}.
\end{equation}
Letting $\nu$ denote the lamination in Corollary~\ref{C:GromovBoundaryLimit}, it follows that $\bnu_\alpha$ and $\bnu_\beta$ are supported on $\nu$.
We will now show that $\bnu_\alpha,\bnu_\beta$ are distinct ergodic measures.


\begin{lemma} \label{Lem:Mutually-Singular}
 We have for $i\geq 0$
\begin{equation*}
\I(\mu, \alpha_{i+1}) \I(\alpha_i, \bnu_\alpha) \emul 1
\qquad\text{and}\qquad
\lim_{i \to \infty} \I(\mu, \alpha_{i+1}) \I(\alpha_i, \bnu_\beta)=0.
\end{equation*}
On the other hand 
\begin{equation*}
\I(\mu, \beta_{i+1}) \I(\beta_i, \bnu_\beta) \emul 1
\qquad\text{and}\qquad
\lim_{i \to \infty} \I(\mu, \beta_{i+1}) \I(\beta_i, \bnu_\alpha)=0.
\end{equation*}
\end{lemma}

\begin{proof}
Since $\bnu_\alpha$ is the limit of $\left\{ \frac{ \alpha_k}{\I(\mu,\alpha_k)} \right\}$, we may apply Equations \eqref{Eq:n} and \eqref{Eq:markn} to conclude that,
for $k$ much larger than $i$, we have 
\begin{align*}
\I(\mu, \alpha_{i+1}) \I(\alpha_i, \bnu_\alpha)
&\emul\I(\mu, \alpha_{i+1}) 
    \, \I \!\left(\alpha_i, \frac{\alpha_k}{\I (\mu, \alpha_k)}\right) \\
&\emul \frac{\prod_{j=1}^{i} 2n_j \, 
        \prod_{j=i+1}^{k-1} 2n_j}{\prod_{j=1}^{k-1} 2n_j} =1
\end{align*}
which proves the first equation. 

Next, observe that by assumption on the sequence $\{r_i\}$ we have
\[ n_j = r_{2j-1} \leq \epsilon r_{2j} = \epsilon m_j\]
for all $j \geq 1$, where $\epsilon < \tfrac12$.
Combining this with Equations \eqref{Eq:m}, \eqref{Eq:markn}, and \eqref{Eq:markm} we have
\begin{align*}
\lim_{i \to \infty} \I(\mu, \alpha_{i+1}) \I(\alpha_i, \bnu_\beta)
&= \lim_{i \to \infty} \lim_{k \to \infty} 
\I(\mu, \alpha_{i+1}) \,\I\!\left(\alpha_i, \frac{\beta_k}{\I(\mu, \beta_k)}\right) \\
&\emul \lim_{i \to \infty} \lim_{k \to \infty} 
   \frac{\prod_{j=1}^{i} 2n_j \, 
    \prod_{j=i+1}^{k-1} 2m_j}{\prod_{j=1}^{k-1} 2m_j} \\
& =  \lim_{i \to \infty} \frac{\prod_{j=1}^{i} 2n_j }{\prod_{j=1}^{i} 2m_j } \leq \lim_{i\to \infty} \epsilon^i = 0. 
\end{align*} 
This proves the first half of the proposition.  The second half can be proved 
in a similar way.
\end{proof}


\begin{corollary} \label{Cor:NUE}
The measured laminations $\bnu_\alpha$ and $\bnu_\beta$ are mutually singular 
ergodic measures on $\nu$. In particular,  $\nu$ is not uniquely ergodic. 
\end{corollary}
\begin{proof}
\lemref{Mutually-Singular} implies
\begin{equation}\label{Eq:distinct}
\frac{\I(\alpha_i,\bnu_\alpha)}{\I(\alpha_i,\bnu_\beta)}\to \infty \quad \text{ while }
\quad \frac{\I(\beta_i,\bnu_\alpha)}{\I(\beta_i,\bnu_\beta)}\to 0.
 \end{equation}
Thus $\bnu_\alpha,\bnu_\beta$ are not scalar multiples of each other, and $\nu$ is not uniquely ergodic.
For the five-times punctured sphere, the maximal dimension of the (projective) simplex of 
measures on a lamination is one. Therefore there are exactly two distinct ergodic measures
on $\nu$, up to scaling. Writing $\bnu_\alpha$ and $\bnu_\beta$ as weighted sums 
of these two measures, \eqnref{distinct} implies that 
each has zero weight on a different ergodic 
measure.
Thus $\bnu_\alpha$ and $\bnu_\beta$ are themselves distinct 
ergodic measures on $\nu$.
\end{proof}

\begin{proof}[Proof of \thmref{NUE}]
The proof follows  from Sections 3, 4 and 5 as follows.
By \lemref{Quasi-Geodesic}, $\{\gamma_i\}$ is a quasi-geodesic if 
 the elements of the  sequence $\{r_i\}$ from some point on are at least $R$.
 It follows from Theorem \ref{KL} that $\{\gamma_i\}$ converges to a minimal filling lamination $\nu$.  
 Requiring  in addition that  $r_i \leq \epsilon r_{i+1}$ for all $i\geq 1$,  we apply 
 \corref{NUE} to conclude that $\nu$ is not uniquely ergodic. 
\end{proof}

\section{Active intervals and shift} 

Let $X=X_0$ be an $\ep_0$--thick Riemann surface in \Teich space $\T(S)$ 
(for some uniform constant $\ep_0$) where the marking $\mu$ has a uniformly 
bounded length. Let 
$$
\bnu = c_\alpha \, \bnu_\alpha + c_\beta \, \bnu_\beta
$$ 
where $c_\alpha, c_\beta\geq 0$ and $c_\alpha + c_\beta > 0$. We have from \eqnref{definition} 
that $\I(\bnu_\alpha, \mu)= \I(\bnu_\beta, \mu) = 1$. Scale  the pair $(c_\alpha,c_\beta)$
so that  $\Hyp_X(\bnu) = 1$. The length of any lamination in $X$
is comparable to its intersection number with any bounded length marking
(see, for example, \cite[Proposition 3.1]{rafi:BC}). Hence, 
\[
c_\alpha + c_\beta = c_\alpha \I( \bnu_\alpha, \mu) + c_\beta \I(\bnu_\beta, \mu) 
=  \I(\bnu, \mu) \emul \Hyp_X(\bnu) = 1.
\]

Let $g= g(X, \bnu)$ be the unique \Teich geodesic ray $g \from \RR_+ \to \T(S)$ 
starting at $X$ in the direction $\bnu$. Let $(X_t,q_t)$ be the associated ray of 
quadratic differentials. Let $v_t(\param)$ and $h_t(\param)$ denote the vertical and 
the horizontal length of a curve at $q_t$. Recall that, for any curve $\gamma$, 
\[
v_t(\gamma) = e^{-t} v_0(\gamma)
\qquad\text{and}\qquad
h_t(\gamma) = e^t v_0(\gamma).
\]
In particular, for any times $s$ and $t$. 
\begin{equation} \label{E:vhconstant}
v_{s}(\gamma) h_{s}(\gamma)=v_{t}(\gamma) h_{t}(\gamma).
\end{equation}
We say a curve $\gamma$ is \emph{balanced} at time $t_0$ if
\[
v_{t_0}(\gamma) = h_{t_0}(\gamma).
\]
From basic Euclidean geometry, we know that for every curve $\gamma$, 
\[
\min \big[ v_t(\gamma), h_t(\gamma )\big]  \leq
\ell_t(\gamma) \leq  v_t(\gamma) + h_t(\gamma ).
\]
Hence, if $\gamma$ is balanced at time $t_0$ then
\[
\ell_t(\gamma) \emul h_t(\gamma) \qquad\text{for}\qquad t\geq t_0
\]
and
\[
\ell_t(\gamma) \emul v_t(\gamma) \qquad\text{for}\qquad t\leq t_0
\]

From \lemref{twist2} and using the results in \cite{rafi:SC, rafi:HT} 
we know that each curve $\alpha_i$ or $\beta_i$ is short at some point along 
this geodesic ray. In fact, the minima of the hyperbolic
length of these curves are independent of values of $c_\alpha$ and $c_\beta$. 
Here, we summarize the consequences of the theory relevant to our current setting:

\begin{proposition} \label{Prop:Intervals}
For $i$ large enough, there are intervals of time $I_i, J_i \subset \RR_+$ 
and a constant $M$ so that the following holds. 
\begin{enumerate}
\item For $t \in I_i$, $q_t$ contains a flat cylinder $F_i$ with core
curve $\alpha_i$ of modulus larger than $M$. Similarly, for $t \in J_i$,
 $q_t$ contains a flat cylinder $G_i$ with core
curve $\beta_i$ of modulus larger than $M$. As a consequence, the intervals 
are disjoint whenever the associated curves intersect. 
\item The intervals $I_i$ appear in order along $\RR_+$, as do the intervals $J_i$.
\item \label{Ppart:interval length} $|I_i| \eadd \log n_i$ and $|J_i| \eadd \log m_i$. 
\item The curve $\alpha_i$ is balanced at the 
midpoint $a_i$ of $I_i$, and $\beta_i$ is balanced at the midpoint $b_i$ of $J_i$.
The flat lengths of $\alpha_i$ and $\beta_i$ take their respective minima at the midpoints of their respective intervals.  In fact
$$
\ell_t(\alpha_i) \emul  \ell_{a_i}(\alpha_i) \cosh (t-a_i)  \quad \mbox{ and } \quad \ell_t(\beta_i) \emul  \ell_{b_i}(\beta_i) \cosh (t-b_i)
$$
\item The hyperbolic length at the balanced time depends only on the 
topological type $\nu$. In fact, for all values of $c_\alpha$ and $c_\beta$, 
$$
\Hyp_{a_i}(\alpha_i) \emul \frac 1{n_i}
\qquad\text{and}\qquad
\Hyp_{b_i}(\beta_i) \emul \frac 1{m_i}.
$$
\end{enumerate}
\end{proposition}

Denote the left and right endpoints of $I_i$ by $\ua_i$ and  $\ba_i$, respectively.  Similarly denote and the endpoints of $J_i$ by $\ub_i$ and $\bb_i$.   As a corollary of Proposition~\ref{Prop:Intervals}(3) we have
\begin{corollary} \label{C:endpointvalues}  For all sufficiently large $i$ we have
\[
\ua_i \eadd a_i - \frac{\log n_i}{2}
\qquad\text{and}\qquad
\ba_i \eadd  a_i + \frac{\log n_i}{2}. 
\]
and
\[
\ub_i \eadd b_i - \frac{\log m_i}{2}
\qquad\text{and}\qquad
\bb_i \eadd  b_i + \frac{\log m_i}{2}.
\]
\end{corollary}

Using Proposition~\ref{Prop:Intervals} and our estimates on intersection numbers, we can estimate the values of $a_i$ and $b_i$ (and hence also of $\ua_i,\ba_i,\ub_i,\bb_i$).

\begin{lemma} \label{L:interval locations}
If $c_\alpha > 0$, then for $i$ sufficiently large (depending on $c_\alpha$) we have
\[
a_i \eadd \sum_{j=1}^{i-1} \log 2n_j + \frac {\log n_i}2  - \frac {\log c_\alpha}2. 
\]
If $c_\alpha = 0$, then for $i$ sufficiently large we have
\[
a_i\eadd  \sum_{j=1}^{i-1} 
   \log 2n_j + \frac {\log n_i}2 + \frac 12 \sum_{j=1}^{i}\log \frac{m_j}{n_j}.
\]
If $c_\beta > 0$, then for $i$ sufficiently large (depending on $c_\beta$) we have 
\[
b_i\eadd \sum_{j=1}^{i-1} \log 2m_j + \frac {\log m_i}2- \frac {\log c_\beta}2.
\]
If $c_\beta = 0$, then for $i$ sufficiently large we have
\[
b_i\eadd  \sum_{j=1}^{i-1} \log 2m_j + 
   \frac {\log m_i}2 + \frac 12\sum_{j=1}^{i}\log \frac{n_{j+1}}{m_j}.
\]
\end{lemma}
The lemma tells us that for $c_\alpha > 0$, as $c_\alpha$ gets smaller, the balance times $a_i$ are shifted more and more
to the right compared with when $c_\alpha=1$. However, the amount of shift is fixed 
for a fixed non-zero value of $c_\alpha$ and large enough $i$.  On the other hand, the shift tends to infinity as $i \to \infty$ when $c_\alpha = 0$.  See 
\figref{times_alpha}. This phenomenon is responsible for the different behavior
of \Teich geodesics associated to different values $c_\alpha$ and $c_\beta$. 
\begin{figure}
    \begin{center}
    \centerline{\includegraphics{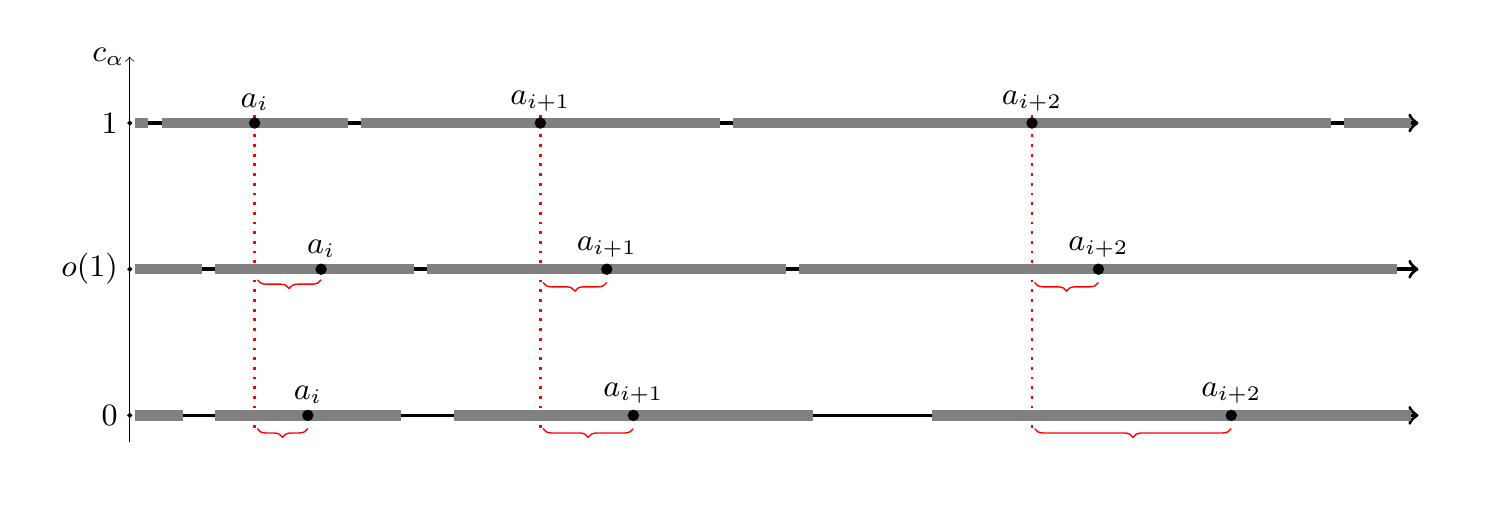}}
    \end{center}
    \caption{The position of intervals $I_i$ depending on the value of $c_\alpha$. 
    When $c_\alpha$
    is small,  balance times $a_i$ are shifted forward. The shift is fixed for a fixed non-zero
    value of $c_\alpha$, and grows as $c_\alpha$ decreases. When $c_\alpha=0$, the shift 
    goes to infinity with $i$. Similar picture holds for intervals $J_i$. }
    \label{Fig:times_alpha}
    \end{figure} 
\begin{proof}
We prove the estimates for $a_i$.  The proofs of the estimates for $b_i$ are similar.
First observe that by Proposition~\ref{Prop:Intervals}(1)--(3), $a_i, b_i$ must tend to infinity.

Since $0\leq a_i$, $v_0(\alpha_i) \emul \ell_0(\alpha_i)$.
Also, since $X_0$ is $\ep_0$--thick, the flat lengths in $q_0$ and the hyperbolic
lengths in $X_0$ are comparable (see \cite{rafi:TT} for a general discussion). 
Hence,
\begin{equation} \label{Eq:v_0}
 v_0(\alpha_i)  \emul \ell_0(\alpha_i) \emul \Hyp_0(\alpha) \emul \I(\alpha_i, \mu) 
\emul \prod_{j=1}^{i-1} 2n_j. 
\end{equation}
where the last equality follows from Corollary~\ref{C:intersectionswithmu}.
By definition of the horizontal length we have
\[
h_0(\alpha_i) \emul \I(\alpha_i, \bnu).
\]
Therefore, combining this all with \eqref{E:vhconstant} we have
\begin{eqnarray} \notag  v_{a_i}(\alpha_i)^2 & = &  
v_{a_i}(\alpha_i) h_{a_i}(\alpha_i)  = v_0(\alpha_i) h_0(\alpha_i)  \emul  \I(\alpha_i, \mu)  \, \I(\alpha_i, \bnu)\\ & \emul & \I(\alpha_i, \mu) 
[c_\alpha \I(\alpha_i, \bnu_\alpha)+ c_\beta \I(\alpha_i, \bnu_\beta)] \label{E:vaisquared}
\end{eqnarray}

We now divide the proof into the two cases specified by the Lemma.\\

\noindent {\bf Case 1.} $c_\alpha > 0$.
\begin{proof}
According to Lemma~\ref{Lem:Mutually-Singular}, $\I(\alpha_i,\bnu_\alpha) \emul \frac{1}{\I(\mu, \alpha_{i+1})}$ while
\[ 0\leq \lim_{i \to \infty} \I(\mu,\alpha_i)\I(\alpha_i,\bnu_\beta) \leq \lim_{i \to \infty} \I(\mu,\alpha_{i+1})\I(\alpha_i,\bnu_\beta) \to 0.\]
Thus, for $i$ sufficiently large (depending on $c_\alpha$) we have
\[ \I(\alpha_i, \mu) [c_\alpha \I(\alpha_i, \bnu_\alpha)+ c_\beta \I(\alpha_i, \bnu_\beta)] \emul  \frac {c_\alpha  \I(\alpha_i, \mu)}{\I(\mu, \alpha_{i+1})}\]

Combining this with \eqref{E:vaisquared} and Corollary~\ref{C:intersectionswithmu} we have
\[ v_{a_i}(\alpha_i)^2  \emul  \I(\alpha_i,\mu)[c_\alpha \I(\alpha_i, \bnu_\alpha)+ c_\beta \I(\alpha_i, \bnu_\beta)] \emul   \frac {c_\alpha \I(\alpha_i, \mu)}{\I(\mu, \alpha_{i+1})}
\emul  c_\alpha \frac{\prod_{j=1}^{i-1} 2n_j}{\prod_{j=1}^{i} 2n_j} 
  = \frac{c_\alpha}{2n_i}.
\]
Thus $v_{a_i}(\alpha_i) \emul \sqrt{\frac{c_\alpha}{2n_i}}$.  Plugging this and \eqref{Eq:v_0} into Proposition~\ref{Prop:Intervals}(4) we obtain
\[ a_i = \log \frac{v_0(\alpha_i)}{v_{a_i}(\alpha_i)}
\eadd \sum_{j=1}^{i-1} \log 2n_j + \frac {\log n_i}2  - \frac {\log c_\alpha}2\]
as required.
\end{proof}

\noindent {\bf Case 2.}  $c_\alpha = 0$.
\begin{proof}
From the definition of $\bnu_\beta$ and Corollary~\ref{C:intersectionswithmu} we have
$$
\I(\alpha_i,\bnu)=\I(\alpha_i,c_{\beta}\bnu_{\beta})=
\lim_{k\to\infty}\frac{c_{\beta}\I(\alpha_i,\beta_k)}{\I(\mu,\beta_k)}\emul
    \frac{1}{\prod_{j=1}^{i}2m_j},
$$
Combining this with (\ref{E:vaisquared}), and appealing to Corollary~\ref{C:intersectionswithmu} we have
\[ v_{a_i}(\alpha_i)^2
   \emul \I(\alpha_i,\mu)i(\alpha_i,c_\beta\bnu_\beta)
   \emul \frac{\prod_{j=1}^{i-1}2n_j}{\prod_{j=1}^{i}2m_j}
                       = \frac{\prod_{j=1}^{i-1}n_j}{2\prod_{j=1}^{i}m_j} \]
As in the previous case, we can plug this and \eqref{Eq:v_0} into Proposition~\ref{Prop:Intervals}(4) to obtain
\begin{eqnarray*} a_i & = & \log\frac{v_0(\alpha_i)}{v_{a_i}(\alpha_i)}
 \eadd \sum_{j=1}^{i-1}\log{2n_j}
+\frac{1}{2}\left[\sum_{j=1}^{i}\log{m_j} -\sum_{j=1}^{i-1}\log{n_j}\right]\\
& \eadd & \sum_{j=1}^{i-1} 
   \log 2n_j + \frac {\log n_i}2 + \frac 12 \sum_{j=1}^{i}\log \frac{m_j}{n_j}.
\end{eqnarray*}
\end{proof}
This takes care of both cases for $c_\alpha$ and completes the proof.
\end{proof}

\section{Growth conditions}
Our assumptions on $\{r_i\}$ in \eqref{E:assumptions-on-r_i} are equivalent to the requirement that the sequences $\{n_i\}$ and $\{m_i\}$ satisfy the following conditions for all $i>0$
\begin{equation}\label{Eq:exponential}
n_1 \geq \frac1\epsilon > 2, \quad\frac{n_i}{m_{i-1}}\geq \frac1\epsilon > 2 \quad\text{and}\quad \frac{m_i}{n_i}\geq\frac1\epsilon > 2.
\end{equation}
After introducing some additional conditions on these sequences, we will investigate the relation
between  intervals $I_i$ and $J_i$. 
\begin{definition} \label{Def:Growth} 
We define the following conditions on $\{n_i\}$ and $\{m_i\}$. 

\begin{equation}
\prod_{j=1}^i\frac{n_j}{m_{j-1}} \leq \prod_{j=1}^{i}\frac{m_j}{n_j} \leq \prod_{j=1}^{i+1}\frac{n_{j}}{m_{j-1}}, \text{  for all big enough } i.\tag{$\calG_1$} \label{tagG1}
\end{equation}
\begin{equation}
\frac{m_{i+1}}{n_{i+1}}=o\left(\prod_{j=1}^{i} \frac{n_j}{m_{j-1}}\right) 
\qquad\text{and}\qquad
\frac{n_{i+2}}{m_{i+1}}=o\left(\prod_{j=1}^i \frac{m_j}{n_j}\right) \quad\text{as }i\to \infty.
  \tag{$\calG_2$} \label{tagG2}
\end{equation}
\end{definition}

It is not difficult to choose $n_i$ and $m_i$ to satisfy the above
conditions. For example: 

\begin{lemma}\label{Lem:exists} 
There exist $\{n_i\}$ and $\{m_i\}$ satisfying \eqref{Eq:exponential} 
and for which \eqref{tagG1} and \eqref{tagG2} hold.
\end{lemma}
\begin{proof}
Fix any integer $K > \frac1\epsilon$, set $m_0 = 1$, and for $i > 0$ set $m_i = K^{2i}$ and $n_i = K^{2i-1}$.  Then for all $i > 0$, we have
\[ 
n_1 = \frac{n_i}{m_{i-1}} = \frac{m_i}{n_i} = K > \frac1\epsilon,
\]
and hence \eqref{Eq:exponential} is satisfied.  In addition
\[ \prod_{j=1}^i\frac{n_j}{m_{j-1}} = \prod_{j=1}^{i}\frac{m_j}{n_j}= K^i  \leq K^{i+1} = \prod_{j=1}^{i+1}\frac{n_{j}}{m_{j-1}},
\]
and thus \eqref{tagG1} and \eqref{tagG2} are satisfied.
\end{proof}

The following lemma provides some information on the position of the intervals
$I_i$ and $J_i$ when the weights $c_\alpha$ and $c_\beta$ are non-zero.
By the notation $\Ga \ll \Gb$ (where $\Ga$ and $\Gb$ are functions of $i$)
we mean $\Gb-\Ga$ goes to infinity as $i$ approaches infinity.  Recall that $\ua_i < \ba_i$ and $\ub_i < \bb_i$ are the endpoints of $I_i$ and $J_i$, respectively.

\begin{figure}
    \begin{center}
    \centerline{\includegraphics{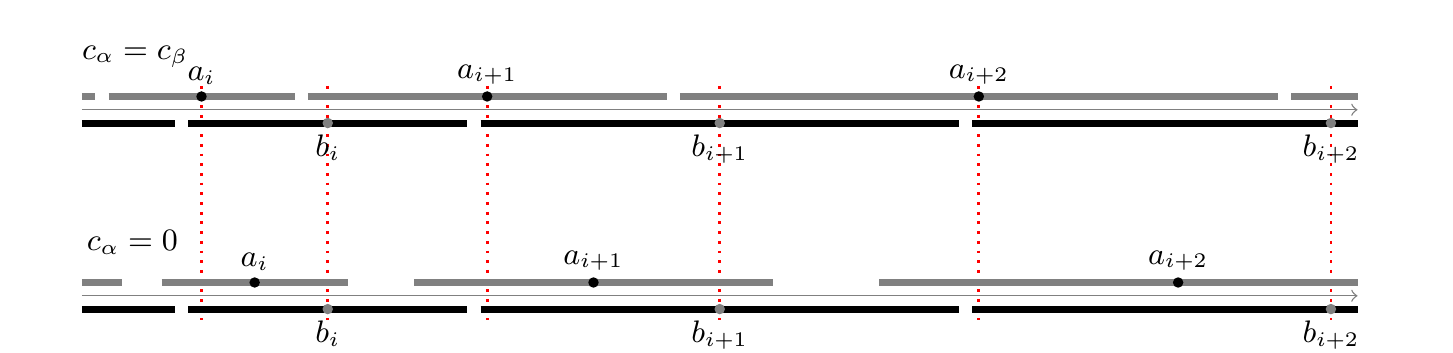}}
    \end{center}
    \caption{Relative position of intervals $I_i$ and $J_i$ depending on $c_\alpha$ and $c_\beta$.
    }
    \label{Fig:times}
    \end{figure} 
    
\begin{lemma} \label{Lem:Times_non_zero} 
Suppose that  $\{n_i\}$ and $\{m_i\}$ satisfy \eqref{tagG1}. Then for any
non-zero value of $c_\alpha$ and $c_\beta$ and  large enough $i$
(depending on $c_\alpha$ and $c_\beta$), we have 
$$
 \ua_i\ll \bb_{i-1} < \ub_{i} \ll  \ba_{i} < \ua_{i+1}\ll \bb_i
$$
and $$
\ba_i \ladd b_i+\frac 1 2 \log\frac{c_\beta}{c_\alpha}
\qquad\text{and}\qquad
 \bb_i \ladd a_{i+1}+\frac 1 2 \log\frac{c_\alpha}{c_\beta}.
$$
\end{lemma}

\begin{proof}  The inequalities $\bb_{i-1} < \ub_i$ and $\ba_i < \ua_{i+1}$ follow from Proposition \ref{Prop:Intervals}(1) since $\I(\beta_{i-1},\beta_i) = 2 = \I(\alpha_i,\alpha_{i+1})$.
We first show $\ub_i \ll \ba_i$; the other two inequalities are similar.  From Corollary~\ref{C:endpointvalues} and Lemma~\ref{L:interval locations} we have
\[
\ba_i = \sum_{j=1}^{i-1} \log 2n_j + \log n_i - \frac{\log c_\alpha}2
\quad\text{and}\quad
\ub_i =  \sum_{j=1}^{i-1} \log 2m_j  - \frac {\log c_\beta}2.
\]
Hence, 
\begin{eqnarray*}
\ba_i - \ub_i  &\eadd & \sum_{j=1}^{i-1}\log \frac{n_j}{m_j}+\log{n_i}+\frac 12\log\frac{c_\beta}{c_\alpha}\\
& = &  \sum_{j=1}^i\log \frac{n_j}{m_{j-1}}+\frac12\log\frac{c_\beta}{c_\alpha} \geq  \sum_{j=1}^i 2j \log(4) +\frac12\log\frac{c_\beta}{c_\alpha}.
\end{eqnarray*}
As the right-hand side tends to infinity when $i \to \infty$, so does the left-hand side.  It follows that $\ub_i \ll \ba_i$, as required.

By a similar argument, we have
\[ \bb_i - \ua_{i+1} \eadd \sum_{j=1}^i \log \frac{m_j}{n_j} + \frac12\log\frac{c_\alpha}{c_\beta}.\]
As this also tends to infinity when $i \to \infty$, we have $\ua_{i+1} \ll \bb_i$, and shifting indices $\ua_i \ll \bb_{i-1}$.  This completes the proof of the first part of the lemma.

We will now prove the second assertion which is where we need
the condition \eqref{tagG1}. Since the proofs of both cases are similar, 
we will only show $\ba_i\ladd b_i+\frac 12 \log\frac{c_\beta}{c_\alpha}$. First note 
that we can write 
\[ 
\log m_i=\sum_{j=1}^{i}\log\frac{m_j}{n_j}+\sum_{j=1}^{i}\log\frac{n_j}{m_{j-1}}. 
\]
Hence, from Corollary~\ref{C:endpointvalues} and Lemma~\ref{L:interval locations} we have
\begin{align*}
b_i-\ba_i+\frac 12 \log\frac{c_\beta}{c_\alpha}
  &\eadd\sum_{j=1}^{i}\log \frac{m_j}{n_j}-\frac 1 2 \log m_i\\
  &= \frac 1 2 \sum_{j=1}^{i}\log \frac{m_j}{n_j}-
\frac 1 2 \sum_{j=1}^{i}\log\frac{n_j}{m_{j-1}}\overset{(\mathcal G_1)}{\geq}0.
\end{align*}
This finishes the proof.
\end{proof}

The next lemma deals with the case when either $c_\alpha=0$ or $c_\beta=0$.  
\begin{lemma} \label{Lem:Times_zero}
Suppose that conditions \eqref{tagG1} and \eqref{tagG2} hold. If $c_\alpha=0$, 
then for sufficiently large $i$ we have
$$
  \ua_i\ll \bb_{i-1} < \ub_{i} \ll a_i\ll  b_i\ll \ba_{i} \ll \ua_{i+1}\ll \bb_i.
$$
Similarly, if $c_\beta=0$, then 
$$
  \ub_{i} \ll \ba_{i} \ladd \ua_{i+1}\ll b_i\ll a_{i+1}\ll\bb_i\ll \ub_{i+1}\ll \ba_{i+1}.
 $$
\end{lemma}

\begin{proof}
Suppose that $c_\alpha = 0$, and hence $c_\beta \emul 1$.  Since $i(\beta_{i-1},\beta_i) = 2$, Proposition~\ref{Prop:Intervals}(1) implies $\bb_{i-1} < \ub_i$.

Applying Corollary~\ref{C:endpointvalues}, Lemma~\ref{L:interval locations}, and condition \eqref{Eq:exponential} we have
\begin{align*}
\ba_i - b_i &\eadd  \sum_{j=1}^{i-1}\log \frac{n_j}{m_j}+\log{n_i}+\frac 1 2 \sum_{j=1}^i  \log
\frac{m_j}{n_j} -\frac 1 2 \log m_i\\
&=\frac 1 2 \sum_{j=1}^{i-1}\log\frac {n_j}{m_j}+\frac 1 2 \log n_i
=\frac 1 2 \sum_{j=1}^{i}\log\frac {n_{j}}{m_{j-1}} \geq \frac12 \sum_{j=1}^i 2j,
\end{align*}
and thus $b_i \ll \ba_i$. 
Similarly
\[ \ua_{i+1} - \ba_i \eadd \frac12 \log \frac{m_{i+1}}{n_{i+1}} > \frac{2i+1}{2},\]
and
\[ b_i - a_i \eadd \frac12 \sum_{j=1}^{i-1} \log \frac{m_j}{n_j} > \frac12 \sum_{j=1}^{i-1} 2j+1\]
and hence $\ba_i \ll \ua_{i+1}$ and $a_i \ll b_i$.

Next we prove $\ub_i \ll a_i$.  For this, we appeal to Corollary~\ref{C:endpointvalues} and Lemma~\ref{L:interval locations} again, and write
\[ a_i - \ub_i \eadd \ba_i - b_i + \frac12 \log \frac{n_i}{m_i}.\]
Combining this with the computation for $\ba_i - b_i$ above we obtain
\[ a_i - \ub_i \eadd \frac 1 2 \sum_{j=1}^{i}\log\frac {n_{j}}{m_{j-1}} + \frac12 \log \frac{n_i}{m_i} = \frac12 \left( \sum_{j=1}^i \log \frac {n_j}{m_{j-1}} - \log \frac{m_i}{n_i} \right).\]
By \eqref{tagG2}, $\displaystyle{\frac{m_i}{n_i} = o \left( \prod_{j=1}^{i-1} \frac{n_j}{m_{j-1}} \right)}$ so by \eqref{tagG1}, $\displaystyle{\frac{m_i}{n_i} = o \left( \prod_{j=1}^i \frac{n_j}{m_{j-1}}\right)}$.  It follows that $a_i - \ub_i \to \infty$ as $i \to \infty$, proving $a_i \ll \ub_i$.

To show  $\ua_{i+1}\ll \bb_i$, we write
\begin{align*}
\bb_i-\ua_{i+1} 
&\eadd
 \sum_{j=1}^{i}\log\frac{m_j}{n_j}-\frac 1 2 \sum_{j=1}^{i+1}\log\frac{m_j}{n_j} \\
&=\frac 1 2 \sum_{j=1}^{i}\log\frac{m_j}{n_j}- \frac 1 2 \log \frac{m_{i+1}}{n_{i+1}}
\end{align*}
which similarly goes to infinity with $i$ by \eqref{Eq:exponential}, \eqref{tagG1}, and \eqref{tagG2}.  The only remaining inequality in the first string is $\ua_i \ll \bb_{i-1}$, which follows from this by shifting the index.

The second string of inequalities, when $c_\beta = 0$, are proved similarly.  This finishes the lemma.
\end{proof}

We will also need the following technical statement. 
\begin{lemma}\label{Lem:slow}
Suppose that conditions \eqref{tagG1} and \eqref{tagG2} hold. Then
\begin{enumerate}
\item 
\[
\log m_i=o\left(\prod_{j=1}^{i}\frac{n_j}{m_{j-1}}\right) 
\quad\text{and}\quad
\log n_{i+1}=o\left(\prod_{j=1}^{i}\frac{m_j}{n_j}\right)
\]
\item If $c_\alpha=0$ then 
\[
e^{\ua_{i+1}-b_i}=o\left(\prod_{j=1}^{i}\frac{n_j}{m_{j-1}}\right) \quad\text{as }i\to \infty.
\]
\item If $c_\beta=0$ then 
\[
e^{\ub_{i+1}-a_{i+1}}=o\left(\prod_{j=1}^{i}\frac{m_j}{n_j}\right) \quad\text{as }i\to \infty.
\]
\end{enumerate}
\end{lemma}
\begin{proof}
We will prove the first claim in $(1)$, the second part is proved similarly.
Rewriting $\log m_i$ and applying \eqref{tagG1} and \eqref{tagG2} we have, for large $i$,
$$
\log m_i=\log\frac{m_i}{n_i}+\sum_{j=1}^{i-1}\log\frac{m_j}{n_j}+\sum_{j=1}^{i}\log\frac{n_j}{m_{j-1}} \leq 3 \sum_{j=1}^i \log \frac{n_j}{m_{j-1}}.
$$
Since $\underset{x\to \infty}{\lim}\frac{\log x}{x}=0$
we obtain  
$$\log m_i=o\left(\prod_{j=1}^{i}\frac{n_j}{m_{j-1}}\right).$$
To prove $(2)$, we compute 
\begin{align*}
\ua_{i+1}-b_i
&\eadd \sum_{j=1}^{i}\log\frac{n_j}{m_{j-1}}+\frac 1 2  \sum_{j=1}^{i+1}\log \frac{m_j}{n_j}-\frac 1 2 \log m_i\\
&= \frac 1 2 \sum_{j=1}^{i}\log\frac{n_j}{m_{j-1}} +\frac 1 2\log \frac{m_{i+1}}{n_{i+1}}
=\frac 1 2 \log\left(\frac{m_{i+1}}{n_{i+1}} \prod_{j=1}^{i}\frac{n_j}{m_{j-1}}\right)
\end{align*}
and so using \eqref{tagG2} we have
$$e^{\ua_{i+1}-b_i}=\left(\frac{m_{i+1}}{n_{i+1}}\right)^{\frac 1 2}\left(\prod_{j=1}^{i}\frac{n_j}{m_{j-1}}\right)^{\frac 1 2}=o\left(\prod_{j=1}^{i}\frac{n_j}{m_{j-1}}\right).$$
The proof of $(3)$ is similar.  
\end{proof}

\section{Limit set in Thurston boundary} 

Let $\gamma$ be a curve. We would like to estimate the length of 
$\gamma$ in various points along a \Teich geodesic.  The geometry of a point
$X\in \calT(S)$ is simple since $S$ is a five-times-punctured sphere.
There are two disjoint curves of bounded length and 
the length of $\gamma$ can be approximated by the number of times it crosses 
these curves and the number of  times it twists around them. We separate the 
contribution to the length of $\gamma$ from crossing and twisting around each curve.

Recall that, for a curve $\alpha$ in $X$, $\twist_\alpha(X, \gamma)$ is the number of times 
$\gamma$ twists around $\alpha$ relative to the arc perpendicular to $\alpha$. 
Also $\width_X(\alpha)$ is the width of the collar around $\alpha$ from the 
Collar theorem \cite{buser:GSC}. We have
\[
\width_X(\alpha) \eadd 2\log\frac{1}{\Hyp_X(\alpha)}.
\]
For a curve $\alpha$ that has bounded length or is short in $X$, define
the \emph{contribution to length of $\gamma$ coming from $\alpha$} to be
$$
\Hyp_X(\alpha, \gamma) =
\I(\gamma, \alpha) \Big[ \width_X(\alpha)
+ \twist_\alpha(X, \gamma) \Hyp_X(\alpha) \Big].
$$

The following estimates the hyperbolic length of a curve in terms of the lengths of the shortest curves in $X$.  See \cite{rafi:LT}, for example.
\begin{theorem} \label{T:hyperbolic-contribution}
If $\alpha,\beta$ are the shortest curves in $X$, then
\[
\Big | \Hyp_X(\gamma)  - \Hyp_X(\alpha, \gamma)  - \Hyp_X (\beta, \gamma) \Big| =
O\big(\I(\alpha, \gamma)+ \I(\beta, \gamma) \big).
\]
\end{theorem}

This theorem essentially says that length of each component of restriction of 
$\gamma$ to the standard annulus around $\alpha$ is the above estimate up to
an additive error and the length of each arc outside of the two 
annuli is also universally bounded. The error is a fixed multiple of
the number of intersections which is the number of these components. 

To analyze how the length of a curve changes along a \Teich geodesic, 
it is enough to know what curves are short at any given time and
then to analyze the contribution to lengths coming from each short curve. 
Let $\gamma$ be any closed curve. The following two lemmas
will provide the needed information about the length of short curves
at any time $t$ and the amount that $\gamma$ twists about curves $\alpha_i$ 
and $\beta_i$. 

\begin{lemma}[Theorem 1.3 in \cite{rafi:CM}]\label{Lem:twisting} 
For a fixed $\gamma$ and large enough $i$ we have  
$$
\twist_{\beta_i}(X_t, \gamma) = 
\begin{cases}
0   \pm O\big(1/\Hyp_t(\beta_i )\big)   & t \leq b_i \\
m_i \pm O\big(1/\Hyp_t(\beta_i )\big)  & t \geq b_i
\end{cases}
$$
A similar statement holds for curves $\alpha_i$, twisting numbers $n_i$ and times $a_i$. 
\end{lemma}
Note that the statement makes sense at $t=b_i$ since, as the following lemma recalls,
 at this time the length of $\beta_i$ is $\frac{1}{m_i}$ up to a bounded multiplicative error.
We also have  

\begin{lemma} \label{Lem:length} 
The function $\Hyp_t(\alpha_i)$ obtains its minimum within a uniform
distance of $a_i$ where $\Hyp_{a_i}(\alpha_i) \emul \frac 1{n_j}$. It changes at most 
exponentially fast. There is a time between $\ba_{i-1}$ and $\ua_i$ 
where the lengths of $\alpha_{i-1}$ and $\alpha_i$ are equal and both are 
comparable to $1$. In particular, when $\ba_{i-1} \eadd \ua_i$, for 
$\ua_i \leq t \leq a_i$, we have
\begin{equation} \label{Eq:width}
\width_t(\alpha_i) \ladd t- \underline a_i. 
\end{equation}
A similar statement holds for curves $\beta_i$, twisting numbers $m_i$ and times $b_i$. 
\end{lemma}

\begin{proof}
The minimum length statement is a restatement from \propref{Intervals}.
The fact that length of curve grows at most exponentially fast is due
to Wolpret \cite{wolpert:LS}. It remains to show that the lengths of $\alpha_{i-1}$ and 
$\alpha_i$ are simultaneously bounded at some time 
between $\ba_{i-1} $ and $\ua_i$. This follows from the theory outlined in
\cite{rafi:HT} and its predecessors. 

Let $a_{i-1, i}$ be the first time after $a_{i-1}$ where length of $\alpha_{i-1}$
is comparable to  $1$. 
Then, the length of $\alpha_{i-1}$ will never be short after that 
\cite[Theorem 1.2]{rafi:CM}. 
This means $\twist_{\alpha_{i-1}}(X_{a_{i-1,i}}, \alpha_i)$ is uniformly bounded,
otherwise, $\alpha_{i-1}$ would have to get short again (to do the twisting) 
before $\alpha_i$ gets short (see \cite{rafi:SC}). Simply put, all the twisting around 
$\alpha_{i-1}$ happens during the interval $I_{i-1}$ and the curves
that are short afterwards will never look twisted around $\alpha_{i-1}$.

To summarize, at time $a_{i-1,i}$, the curve $\alpha_{i-1}$ has a
length comparable to 1 and $\alpha_i$, which intersects $\alpha_{i-1}$ twice, 
does not twist around $\alpha_{i-1}$. By \lemref{Times_non_zero} and 
\lemref{Times_zero} the other short curve in the surface is $\beta_i$
that is disjoint from $\alpha_i$. Therefore, $\alpha_i$ also has length comparable 
to 1. 
\end{proof}

\thmref{intro} stated in  the introduction is a direct consequence of the following theorem.
\begin{theorem}\label{limit}
Suppose $\{n_i\}$ and
$\{m_i\}$ satisfy condition \eqref{tagG1}.
Then, if  $c_\alpha$ and $c_\beta$ are non-zero, the limit set in $\mathcal{PML}(S)$ of the 
corresponding ray $g$ is the entire simplex spanned by $\bnu_{\alpha}$ and $\bnu_{\beta}$.
If, in addition \eqref{tagG2} holds, then the limit set of $g$ is the entire simplex 
 for any value of $c_\alpha$ and $c_\beta$.
\end{theorem}

\begin{proof}[Proof of Theorem \ref{limit}] 
It is enough to show that the limit set contains $\bnu_\alpha$. 
Then, because of the symmetry, we will also have that $\bnu_\beta$ is 
in the limit set.  Since the limit set is connected and consists only of 
laminations topologically equivalent to $\nu$, it must be the entire interval. 

We need to find a 
sequence of times $t_i\to\infty$ such that for  any two simple closed curves $\gamma$ and $\gamma'$ 
\[
\frac{\Hyp_{t_i}(\gamma)}{\Hyp_{t_i}(\gamma')}\to 
  \frac{\I(\bnu_\alpha,\gamma)}{\I(\bnu_\alpha,\gamma')}.
\]
 
\subsection*{Case 1.} Assume $c_\alpha\neq 0$ and that
\eqref{tagG1} holds.  If $c_\beta = 0$, then also assume \eqref{tagG2}.
\begin{proof}
We will prove the statement for the sequence of times $\{\ba_i\}$.  By \lemref{Times_non_zero} and
\lemref{Times_zero}, we have $\ba_i \eadd \ua_{i+1}$ and hence, 
by \lemref{length}, $\Hyp_{\ba_i}(\alpha_i)\emul 1$.
Let $\gamma$ be a simple closed curve. We will estimate the hyperbolic length of $\gamma$ at time $\ba_i$. 

Let $i$ be large enough so that the inequalities in \eqref{tagG1} hold and so that
$$
\I(\gamma, \bnu_\alpha) \emul \frac{\I(\gamma, \alpha_i)}{\I(\mu, \alpha_i)}
\qquad\text{and}\qquad
\I(\gamma, \bnu_\beta) \emul \frac{\I(\gamma, \beta_i)}{\I(\mu, \beta_i)}.
$$
Then, applying \eqnref{markn} and \lemref{twisting}, 
\begin{eqnarray}
\Hyp_{\ba_i} (\alpha_i, \gamma) 
 &= &\I(\gamma, \alpha_i) \big[ O(1)+ (n_i \pm O(1))\Hyp_{\ba_i}(\alpha_i) \big]  \label{E:star} \\
 &\emul &  \I(\gamma, \bnu_\alpha) \I(\mu, \alpha_i) n_i \notag \\
 &\emul & \I(\gamma, \bnu_\alpha) \prod_{j=1}^i 2n_j.  \notag
\end{eqnarray}

By \lemref{Times_non_zero} and  \lemref{Times_zero},  we have
$\ub_i\ll\ba_i\ll\bb_i$, therefore $\beta_i$ is short at $\ba_i$. Hence, we need 
to compute the contribution to the length of $\gamma$ from $\beta_i$. 
Note that, depending on whether $c_\beta$ is zero or not, we have:
\[
\ba_i -b_i\ladd\frac 12\log\frac{c_\beta}{c_\alpha}
\qquad\text{or}\qquad 
\ba_i \ll b_i. 
\]
That is, for the purposes of this case, if we consider $\frac{c_\beta}{c_\alpha}$
to be a uniform constant (independent of $i$), we can write
$\ba_i \ladd b_i$, which means $\ba_i$ is in the first half of the interval $J_i$. 
Therefore, by \lemref{twisting} and \lemref{length}
\[
\twist_{\beta_i}(X_{\ba_i} , \gamma)  
\Hyp_{\ba_i}(\beta_i) = O(1) 
\quad \text{and}\quad
\width_{\ba_i}(\beta_i) \ladd \ba_i - \ub_i.
\]
If $c_\beta > 0$, then by Corollary~\ref{C:endpointvalues} and Lemma~\ref{L:interval locations} we have
$$ 
\ba_i - \ub_i \eadd \sum_{j=1}^{i} \log \frac {n_j}{m_{j-1}}+\frac 12\log\frac{c_\beta}{c_\alpha} \ladd \sum_{j=1}^i \log \frac{n_j}{m_{j-1}}
$$
where again we are taking $\frac{c_\beta}{c_\alpha}$ to be a uniform constant in the inequality $\ladd$ here.  Similarly, when $c_\beta = 0$, we have
\[ \ba_i - \ub_i \eadd \frac12 \left( \sum_{j=1}^i \log \frac{n_j}{m_{j-1}} - \log \frac{n_{i+1}}{m_i} \right) < \sum_{j=1}^i \log \frac{n_j}{m_{j-1}}.\]
Since the right-hand side of both of these inequalities tend to infinity, for any $c_\beta$ we have
\[ \width_{\ba_i}(\beta_i) \lmul \sum_{j=1}^i \log \frac{n_j}{m_{j-1}}.\]


Now, appealing to \lemref{length} and these computations, we see that for large enough $i$  
\begin{eqnarray}
\Hyp_{\ba_i} (\beta_i, \gamma) 
& \lmul &  \I(\gamma, \beta_i)  \Big[ \sum_{j=1}^{i} \log \frac {n_j}{m_{j-1}} +O(1) \Big] \label{E:starbeta}\\
& \emul & \I(\gamma, \bnu_\beta) \I(\mu, \beta_i)  \sum_{j=1}^{i} \log \frac {n_j}{m_{j-1}} \notag \\  
&\emul  & \I(\gamma, \bnu_\beta) \prod_{j=1}^{i-1} 2m_j  \sum_{j=1}^{i} \log \frac {n_j}{m_{j-1}}.  \notag 
\end{eqnarray}

Note that for any  curve $\gamma$, the ratio $\I(\gamma, \bnu_\alpha) / \I(\gamma, \bnu_\beta)$ is a fixed
number independent of $i$, and the sum $\sum_{j=1}^{i} \log \frac{n_j}{m_{j-1}}$ 
is negligible compared to the product $\prod_{j=1}^i \frac{n_j}{m_{j-1}}$.
Therefore, dividing the above estimate by \eqref{E:star}, we obtain
$$
\lim_{i \to \infty} 
\frac{\Hyp_{\ba_i}(\beta_i, \gamma)} 
{\Hyp_{\ba_i}(\alpha_i, \gamma) } = 0.
$$
Clearly,  the above holds for any other simple closed curve $\gamma'$, and so does the estimate in \eqref{E:star}.  
Thus by Theorem~\ref{T:hyperbolic-contribution} and appealing to $\displaystyle{ \lim_{i \to \infty} \frac{\Hyp_{\ba_i}(\beta_i, \gamma)} {\Hyp_{\ba_i}(\alpha_i, \gamma) } = 0}$ and \eqref{E:star} we have
\begin{eqnarray*}
\lim_{i \to \infty} \frac{\Hyp_{\ba_i}(\gamma)}{\Hyp_{\ba_i}(\gamma')} &= &  \lim_{i \to \infty} \frac{\Hyp_{\ba_i}(\alpha_i,\gamma) + \Hyp_{\ba_i}(\beta_i,\gamma) + O( \I(\alpha_i,\gamma) + \I(\beta_i,\gamma))}{\Hyp_{\ba_i}(\alpha_i,\gamma') + \Hyp_{\ba_i}(\beta_i,\gamma') + O( \I(\alpha_i,\gamma') + \I(\beta_i,\gamma'))}\\
& = & \lim_{i \to \infty} \frac{\Hyp_{\ba_i}(\alpha_i,\gamma)}{\Hyp_{\ba_i}(\alpha_i,\gamma')}\\
& = & \lim_{i \to \infty}  
  \frac{\I(\gamma, \alpha_i)(O(1) + n_i\Hyp_{\ba_i}(\alpha_i))}{\I(\gamma', \alpha_i)(O(1)+n_i\Hyp_{\ba_i}(\alpha_i))}
= \frac{\I(\gamma, \bnu_\alpha)}{\I(\gamma', \bnu_\alpha)}
\end{eqnarray*}

This implies that 
$$
X_{\ba_i} \to \bnu_\alpha
$$
in $\PML(S)$, completing the proof in this case.
\end{proof}

\subsection*{Case 2.} 
Assume now that  $c_\alpha=0$.

\begin{proof}
As above, let $i$ be big enough so that 
$\{n_i\}$ and $\{m_i\}$ satisfy inequalities in \eqref{tagG1} and \eqref{tagG2} and so that
$$
\I(\gamma, \bnu_\alpha) \emul \frac{\I(\gamma, \alpha_i)}{\I(\mu, \alpha_i)}
\qquad\text{and}\qquad
\I(\gamma, \bnu_\beta) \emul \frac{\I(\gamma, \beta_i)}{\I(\mu, \beta_i)}.
$$
The curve $\alpha_i$ might still be very short at $\ba_i$,  in which case
its contribution to the length of $\gamma$ would not be maximal at $\ba_i$.

Hence we will estimate the length of $\gamma$ at $a_{i,i+1}\in[\ba_i,\ua_{i+1}]$, 
(see \lemref{length}) when the hyperbolic length of $\alpha_i$ and $\alpha_{i+1}$
are both comparable to $1$. By a computation essentially the same as in \eqref{E:star}, 
the contribution from $\alpha_i$ to the length of  the curve $\gamma$ is

\begin{equation} \label{E:starstar}
\Hyp_{a_{i,i+1}} (\alpha_i, \gamma) 
 \emul \I(\gamma, \bnu_\alpha) \prod_{j=1}^i 2n_j.  
\end{equation}
We now estimate the contribution of $\beta_i$ at the time $a_{i,i+1}$. 
Note that
$[\ba_i,\ua_{i+1}]\subset [b_i,\bb_i]$ by \lemref{Times_zero}, so $\beta_i$
is short at this moment and its length is increasing at most exponentially.
But $\Hyp_{b_i}(\beta_i) \emul 1/m_i$. Hence, 
\[
\Hyp_{a_{i,i+1}}(\beta_i)\lmul \frac{e^{\ua_{i+1}-b_i}}{m_i}.
\]
The width of $\beta_i$ is bounded above by $\log m_i$ for any value of $t$.
By \lemref{twisting} and  \eqnref{markm},
 \begin{align*}
\Hyp_{a_{i,i+1}}(\beta_i,\gamma)
& \lmul\I(\gamma,\beta_i)\left(\log m_i+e^{\ua_{i+1}-b_i}\right)\\
& \emul  \I(\gamma, \bnu_\beta) \prod_{j=1}^{i-1} 2m_j \left(\log m_i+e^{\ua_{i+1}-b_i}\right)
\end{align*}
 which, together with \eqref{E:starstar} and \lemref{slow} implies that 
$$
\lim_{i \to \infty} 
\frac{\Hyp_{a_{i, i+1}}(\beta_i, \gamma)} 
{\Hyp_{a_{i, i+1}}(\alpha_i, \gamma) } = 0.
$$
Repeating the argument at the end of the Case 1, 
we conclude that  the projective class of $\bnu_\alpha$ is in the limit set.  
This  completes the proof in this case.
\end{proof}
As both cases exhaust the possibilities, this completes the proof of the theorem.
\end{proof}
  \bibliographystyle{alpha}
  \bibliography{main}

  \end{document}